\documentclass[12pt,oneside]{amsart}
\usepackage{amsmath,amssymb,latexsym,soul,cite,mathrsfs}
\usepackage{color,enumitem,graphicx}
\usepackage[colorlinks=true,urlcolor=blue,
citecolor=green,linkcolor=blue,linktocpage,pdfpagelabels,
bookmarksnumbered,bookmarksopen]{hyperref}
\usepackage[english]{babel}
\usepackage{lineno}

\usepackage[left=2.4cm,right=2.8cm,top=2.5cm,bottom=2.4cm]{geometry}

\usepackage[hyperpageref]{backref}

\numberwithin{equation}{section}
\newtheorem{theorem}{Theorem}[section]
\newtheorem{lemma}[theorem]{Lemma}
\newtheorem{corollary}[theorem]{Corollary}
\newtheorem{proposition}[theorem]{Proposition}

\newtheorem{remark}[theorem]{Remark}
\newtheorem{definition}[theorem]{Definition}

\renewcommand{\epsilon}{\varepsilon}

\renewcommand{\rightarrow}{\to}

\newcommand{\ud}{\mathrm{d}}

\title[Supercritical k-Hessian inequality]{Extremal functions for a supercritical k-Hessian inequality of Sobolev-type}

\author[J.F. de Oliveira]{Jos\'{e} Francisco de Oliveira}
\address[J.F de Oliveira]{\newline\indent Department of Mathematics
\newline\indent 
Federal University of Piau\'{\i}
\newline\indent
 64049-550 Teresina, PI, Brazil}
\email{\href{mailto:jfoliveira@ufpi.edu.br}{jfoliveira@ufpi.edu.br}}
\author[P. Ubilla]{Pedro Ubilla}
\address[P. Ubilla]{\newline\indent Departamento de Matematica
\newline\indent 
Universidad de Santiago de Chile
\newline\indent
Casilla 307, Correo 2, Santiago, Chile}
\email{\href{mailto:pedro.ubilla@usach.cl}{pedro.ubilla@usach.cl}}
\thanks{The second author was partially supported by FONDECYT 1181125, 1161635 and 1171691.}
\begin{document}
\begin{abstract}
Our main purpose in this paper is to investigate a supercritical Sobolev-type inequality for the $k$-Hessian operator acting on $\Phi^{k}_{0,\mathrm{rad}}(B)$, the space of radially symmetric $k$-admissible functions on the unit ball $B\subset\mathbb{R}^{N}$. We also prove both the existence of admissible extremal functions for the associated variational problem and the solvability of a related $k$-Hessian equation with supercritical growth. 
\end{abstract}
\maketitle
\section{Introduction}
\label{section1}
It is well known that Sobolev-type inequalities and the corresponding variational problems play an important role in many branches of mathematics such as analysis, partial differential equations, geometric analysis, and calculus of variations.
The classical Sobolev inequality provides an optimal embedding from the Sobolev space $H^{1}(\Omega)$ into the Lebesgue spaces $L^{p}(\Omega)$ with $p\le 2^{*}=2N/(N-2)$, where $\Omega\subset\mathbb{R}^{N}$, with $N\ge 3$ is a bounded smooth domain. By restricting to the Sobolev space of radially symmetric functions about the origin $H^{1}_{0,\mathrm{rad}}(B)$, where $B$ is the unit ball in $\mathbb{R}^{N}$,  J.M. do \'{O}, B. Ruf, and P. Ubilla in \cite{MR3514752} were able to prove a variant of the Sobolev inequality giving an embedding into 
non-rearrangement invariant spaces $L_{p(x)}(B)$, the variable
exponent Lebesgue spaces, which goes beyond the critical exponent $2^{*}$. Namely, it was proven that
\begin{equation}\label{primary}
\mathcal{U}_{N,\alpha}=\sup\left\{\int_{B}|u|^{2^{*}+|x|^{\alpha}}\mathrm{d} x\; |\; u\in H^{1}_{0,\mathrm{rad}}(B),\; \|\nabla u\|_{L^2(B)}=1 \right\}<\infty
\end{equation}
for all $\alpha>0$. In addition, the supremum in \eqref{primary} is 
attained under the condition
\begin{equation}\label{alphaDoRUUb}
 0<\alpha<\min\left\{N/2, N-2\right\}.
 \end{equation}
As an application, the authors were also able to prove that the following supercritical elliptic equation
\begin{equation}\label{primary-problem}
\left\{\begin{aligned}
&\left.\begin{aligned}
&-\Delta u=|u|^{2^{*}+|x|^{\alpha}-2}u,\\
\end{aligned}\right. &\mbox{in}&\;\; B\\
&\;u=0,&\mbox{on}&\;\;\partial B
\end{aligned}\right.
\end{equation}
admits at least one positive solution for all $0<\alpha<\min\left\{N/2, N-2\right\}$. This, is somewhat surprising, for the nonlinearity term has strictly supercritical growth except in the origin. Based on these results,  the authors in \cite{Cao} were able to show the existence of infinitely many nodal solutions for problem \eqref{primary-problem}.  The inequality \eqref{primary} and its applications have currently captured attention. In recent work \cite{Ngu}, Q.A. Ng\^{o} and V.H. Nguyen proved that the inequality \eqref{primary} and its extremal problem can be extended for higher order Sobolev spaces $H^{m}_{0,\mathrm{rad}}(B), m\ge 1$, while in \cite{DCDS2019}  suitable extension including $W^{1,p}_{0,\mathrm{rad}}(B),p\ge 2$ has been done, motivated by the classical Hardy inequality \cite{Hardy1920}. For more results related to this class of problems, the reader can see \cite{JDE2019, CLDOUN,CZW}.

Now, let us introduce $F_k$,  $1\le k\le N$  be the $k$-Hessian operator defined by  
\begin{equation}\nonumber
F_k[u]=\sum_{1\le i_1<\dots<i_k\le N}\lambda_{i_1}\dots\lambda_{i_k},
\end{equation}
where $\lambda=(\lambda_1, \dots, \lambda_N)$ are the eigenvalues of the real symmetric Hessian matrix $D^{2}u$ of a function $u\in C^{2}(\Omega)$. Alternatively,
$F_k[u]$ is the sum of all $k \times k$ principal minors of the Hessian matrix $D^{2}u$, which coincides with the Laplacian $F_1[u]=\Delta u$  if $k=1$ and the Monge-Amp\`{e}re operator $F_{N}[u]=\det(D^2u)$ if $k=N$.

The main purpose in this work is to provide a version of the inequality \eqref{primary} for the $k$-Hessian operator $(k\ge 1)$ and to analyze the existence of optimizer for related extremal problem. 

Before introducing our results, let us  briefly explain the range and ingredients presented in the paper. For $k=1,\dots, N$, the $k$-Hessian operators have a divergence structure (see for instance \cite{MR1051232}), but their variational structure is not compatible with any classical Sobolev space $W^{m, p}_{0}(\Omega),m\ge 1, p\ge 2$ in the fully nonlinear range $k=2,\dots, N$, see for instance \cite{Wang2}. In addition, operators $F_k$, $k=2,\dots, N$ are not elliptic on whole space $C^{2}(\Omega)$. In fact, we must consider the $k$-admissible function space  $\Phi^{k}_0(\Omega)$ proposed by L. Caffarelli, L. Nirenberg, and J. Spruck in the pioneer work  \cite{MR3487276}, which is the subspace of the functions $u\in C^{2}(\Omega)$ vanishing on $\partial\Omega$ such that $F_{j}[u]\ge 0,\; j=1,\dots,k$. Although the $F_k$,  $k=2,\dots, N$ are elliptic when restricted to $\Phi^{k}_0(\Omega)$, we must now prove that our maximizer is a smooth function. We also observe that we cannot use the variational theory directly on $\Phi^{k}_0(\Omega)$ because we don't know the behavior of a functional $I:\Phi^{k}_0(\Omega)\rightarrow\mathbb{R}$ near of the boundary. To overcome this, a descent gradient flow of $I$ has been employed for some authors, see for instance \cite{TianWang,Wang2}. In this work, by using a Hardy-type inequality \cite{Opic}, we are able to carry out the analysis in a suitable weighted Sobolev space which was already employed in \cite{MR1051232,MZ2020,JDE2019}, see Section~\ref{sec2} below. Note that the space $\Phi^{k}_0(\Omega)$ has been used by several authors to study the $k$-Hessian equation. For instance, existence, multiplicity, uniqueness, and asymptotic behavior of radially symmetric $k$-admissible solutions of the $k$-Hessian equation were investigated in \cite{Dai, Wei,Wei2,Justino-Vicente,MR1400797,MR1051232}; while details on space $\Phi^{k}_0(\Omega)$ and more general results can be found in \cite{MR3487276,Wang2,Labutin,ChouWang,MR1368245, TianWang} and references therein. 

 As observed in \cite{Wang2}, the expression
\begin{equation}\label{admissible-norm}
\|u\|_{\Phi_0^k}= \left(\int_{\Omega}(-u)F_{k}[u]\ud x\right)^{\frac{1}{k+1}}, \;\; u\in \Phi^{k}_{0}(\Omega
)
\end{equation}
defines a norm on  $\Phi^{k}_{0}(\Omega
)$. In addition, the following Sobolev type inequality holds: there exists a constant $C = C(N,k,p,\Omega)$ such that
\begin{equation}\label{S-inequality}
\|u\|_{L^{p}(\Omega)} \le C \, \|u\|_{\Phi_0^k}, \forall \ u \in \Phi_0^k(\Omega)
\end{equation}
for any $1\le p\le k^*$, where 
$$
k^* = \frac{N(k+1)}{N-2k},\;\;\;  1\le k<\frac{N}{2}
$$ 
is the optimal exponent of the inequality \eqref{S-inequality}.

Our first main result reads as follows:
\begin{theorem}\label{thm1} Let $\alpha>0$ be real number and assume $1\le k< N/2$. Then
\begin{equation}\label{mainS}
\sup \left\{\int_{B}|u(x)|^{k^{*}+|x|^{\alpha}}\ud x\;\;|\;\; u\in\Phi^{k}_{0,\mathrm{rad}}(B)\;,\; \|u\|_{\Phi^{k}_{0}}=1\right\}<\infty,
\end{equation}
where $\Phi^{k}_{0,\mathrm{rad}}(B)$ is the subspace of radially symmetric functions in $\Phi^{k}_0(B)$.
\end{theorem}
The above result represents the counterpart of \eqref{primary} to the fully nonlinear case $F_k,k\ge 2$ (recall $F_1[u]=\Delta u$).

In the same line of \cite{MR3514752}, one can see that Theorem~\ref{thm1} ensures the continuous embedding of the $k$-admissible function space $\Phi^{k}_{0,\mathrm{rad}}(B)$ into the variable exponent Lebesgue space $L_{k^*+|x|^{\alpha}}(B)$. Precisely,
\begin{corollary}\label{coro1} Let $1\le k<N/2$ and $\alpha>0$. Then the following embedding is continuous 
\begin{equation}\nonumber
\Phi^{k}_{0,\mathrm{rad}}(B)\hookrightarrow L_{k^*+|x|^{\alpha}}(B),
\end{equation}
where  $L_{k^*+|x|^{\alpha}}(B)$ denotes the variable exponent Lebesgue space defined by
\begin{equation}\nonumber
L_{k^*+|x|^{\alpha}}(B)=\left\{u:B\rightarrow\mathbb{R}\;\;\mbox{is mensurable}\;\; | \;\; \int_{B}|u(x)|^{k^*+|x|^{\alpha}}\ud x<+\infty\right\},
\end{equation}
with the norm
\begin{equation}\nonumber
\|u\|_{L_{k^{*}+|x|^{\alpha}}(B)}=\inf\left\{\lambda>0\;\; 
| \;\; \int_{B}\left|\frac{u(x)}{\lambda}\right|^{k^*+|x|^{\alpha}}\ud x\le 1 \right\}.
\end{equation}
\end{corollary}
On the attainability, we can prove the following:
\begin{theorem}\label{thm2}
For $\alpha>0$ be real number and $1\le k< N/2$, we set
\begin{equation}\label{max-problem}
\mathcal{U}_{k, N,\alpha}=\sup \left\{\int_{B}|u(x)|^{k^{*}+|x|^{\alpha}}\ud x\;\;|\;\; u\in\Phi^{k}_{0,\mathrm{rad}}(B)\;,\; \|u\|_{\Phi^{k}_{0}}=1\right\}.
\end{equation}
Then, $\mathcal{U}_{k,N,\alpha}$ is attained provided that
$0<\alpha
<(N-2k)/k.$
\end{theorem}
In comparison with the attainability of \eqref{primary}, which was already obtained in \cite{MR3514752}, Theorem~\ref{thm2} extends that for the fully nonlinear situation $k\ge 2$. In addition, even for $k=1$, the range of $\alpha$ (cf. \eqref{alphaDoRUUb}) is improved for $0<\alpha<N-2$ (see also \cite{Ngu}) and our extremal is an admissible function.

Finally, we study the existence of radially symmetric $k$-admissible solutions of the $k$-Hessian equation involving supercritical growth.
\begin{theorem}\label{thm3}
 Suppose $1\le k<N/2$ and $0<\alpha<(N-2k)/k$, then equation
 \begin{equation}\label{ourp}
\left\{\begin{aligned}
&\left.\begin{aligned}
&F_k[u]=(-u)^{k^{*}+|x|^{\alpha}-1}\\
&u<0 \\
\end{aligned}\right\}&\mbox{in}&\;\; B\\
&\;u=0&\mbox{on}&\;\;\partial B
\end{aligned}\right.
\end{equation}
admits at least one radially symmetric $k$-admissible solution.
\end{theorem}
To prove the existence of a radially symmetric solution of Problem \eqref{ourp} is equivalent to finding a solution of the following boundary value problem
\begin{equation}\label{Kourpradial}
\left\{\begin{aligned}
&\left.\begin{aligned}
&\mathrm{C}^{k}_{N}\left(r^{N-k}(w^{\prime})^{k}\right)^{\prime}=Nr^{N-1} (-w)^{k^{*}+r^{\alpha}-1}\\
&w<0 \\
\end{aligned}\right\}&\mbox{in}&\;\; (0,1)\\
&\;w(1)=0, \; w^{\prime}(0)=0
\end{aligned}\right.
\end{equation}
where $w(r)=u(|x|)$, and $\mathrm{C}^{m}_{n}=n!/((n-m)!m!)$ is the combinatorial constant. See \cite{MR1400797,Dai,Wei} for more details. It was recently shown in \cite{JDE2019} that the equation \eqref{Kourpradial}  admits at least one solution $w\in C^{2}(0,1)$, if $0<\alpha<\min\left\{N/(k+1), (N-2k)/k\right\}$. Theorem~\ref{thm3} improves the previous one in the sense that the solution is a $k$-admissible function and the existence condition is sharpened for $0<\alpha<(N-2k)/k$.

The rest of this paper is arranged as follows. In Section~\ref{sec2}, we show Theorem~\ref{thm1} and its consequence the Corollary~\ref{coro1}.  Section~\ref{sec3} is devoted to the study of an auxiliary extremal problem. The proof of Theorem~\ref{thm2} is given in Section~\ref{sec4}. In Section~\ref{sec5}, we ensure the existence of a radially symmetric $k$-admissible solution of the nonlinear equation \eqref{ourp}. 

\section{The inequality: Proof of Theorem~\ref{thm1}}
\label{sec2}
This section is devoted to prove Theorem~\ref{thm1}. In order to use variational techniques, we are going to employ the weighted Sobolev space $X_R=X_R^{1,k+1}$ which consists of all functions $v\in AC_{\mathrm{loc}}(0,R)$ satisfying
$$
\lim_{r\rightarrow R}v(r)=0,\;\;\; \int_{0}^{R}r^{N-k}|v^{\prime}|^{k+1}\mathrm{d}r<\infty\;\;\;\mbox{and}\;\;  \int_{0}^{R}r^{N-1}|v|^{k+1}\mathrm{d}r<\infty,
$$
where $AC_{\mathrm{loc}}(0, R)$ is the set of all locally absolutely continuous functions on interval $(0,R)$.

If  $0<R<\infty$, then $X_R$ is a Banach space endowed with the gradient norm
\begin{equation}\label{fullnorm}
\|v\|_{X_R}=\left(\omega_{N,k}\int_{0}^{R}r^{N-k}|v^{\prime}|^{k+1}\mathrm{d}r\right)^{\frac{1}{k+1}},
\end{equation} 
where $\omega_{N,k}=(\omega_{N-1}\mathrm{C}^{k}_{N})/N$ is a normalising constant, and  $\omega_{N-1}$ represents the area of the unit sphere in $\mathbb{R}^N$.
For more details on the weighted Sobolev space above as well as its applications, we refer the reader to \cite{JDE2019,QJM2020,MZ2020, DCDS2019} and to the recent work \cite{QJM2020} where an inherent discussion has been done. 

As a byproduct of a Hardy type inequality, which is essentially due to A. Kufner and B. Opic \cite{Opic} (see also \cite{Clement-deFigueiredo-Mitidieri}), the following continuous embedding holds:
\begin{equation}\label{Ebeddings}
    X_R \hookrightarrow L^q_{N-1}, \quad \mbox{if}\quad q \in \left.\left[1, k^{*}\right.\right]\;\;\;\; \mbox{(compact if}\;\;q<k^{*}),
\end{equation}
where $L^q_{N-j}=L^q_{N-j}(0,R),q\ge 1, j\in\left\{1,2,\cdots,k\right\}$ is the weighted Lebesgue space composed by all measurable functions $v$ on $(0,R)$ such that
\begin{equation}\nonumber
 \|v\|_{L^q_{N-j}}=
\left(\int_0^R r^{N-j}|v|^q\,\mathrm{d}r
\right)^{\frac{1}{q}}<+\infty.
\end{equation}
Let $u\in\Phi^{k}_{0,\mathrm{rad}}(B)$ be arbitrary. We recall that (see for instance \cite{Wang2})
\begin{equation}\label{radial-norm}
\|u\|_{\Phi^{k}_{0}}=\left(\omega_{N,k}\int_{0}^{1}r^{N-k}|u^{\prime}|^{k+1}\ud r\right)^{\frac{1}{k+1}}.
\end{equation}
Hence, setting $v(r)=u(x)$, $r=|x|$,  we clearly have  $v\in X_1$ satisfying
\begin{equation}\label{FX}
\int_{B}|u|^{k^{*}+|x|^{\alpha}}\mathrm{d}x=\omega_{N-1}\int_{0}^{1}r^{N-1}|v|^{k^{*}+r^{\alpha}}\mathrm{d}r,
\end{equation}
and
\begin{equation}\label{NormX}
\|u\|_{\Phi^{k}_0}=\|v\|_{X_1}.
\end{equation}
Taking into account \eqref{FX} and \eqref{NormX}, we are able to transfer our problem to weighted Sobolev space $X_1$. Indeed, we will first investigate the variational problem
\begin{equation}\label{supremumX} 
\mathcal{V}_{k,N,\alpha}=\sup_{\|v\|_{X_1}=1}\int_{0}^{1}r^{N-1}|v|^{k^{*}+r^{\alpha}}\mathrm{d}r.
\end{equation}
Note that $\mathcal{U}_{k,N,\alpha}\le \omega_{N-1}\mathcal{V}_{k,N,\alpha}$, but the converse inequality doesn't hold. However, in \cite[Propostion~2.4]{JDE2019} has proved that $\mathcal{V}_{k,N,\alpha}$
is finite. Of course, this implies Theorem~\ref{thm1} holds. Next, we are going to present another proof of Theorem~\ref{thm1} which works for more general situation, see Lemma~\ref{super-bounded} below. It will be necessary to obtain regularity estimate for extremal function in Section~\ref{sec4} below. 
\subsection{Proof Theorem~\ref{thm1}}
We start with the following radial type Lemma (cf. \cite{Strauss77}).
\begin{lemma}\label{r-typeLemma} Assume $1\le k<N/2$. Then,  for any $0<r\le 1$
\begin{equation}\label{2.X}
|v(r)|\le \left[\overline{c}(r^{-\frac{N-2k}{k}}-1)\right]^{\frac{k}{k+1}}\|v\|_{X_1},\;\;\forall\; v\in X_1,
\end{equation}
where $\overline{c}$ is a positive constant depending only on $N$ and $k$.
\end{lemma}
\begin{proof}
Let $v\in X_1$ be arbitrary. Using 
the H\"{o}lder inequality one has
\begin{equation}\nonumber
\begin{aligned}
|v(r)|&\le  \int_{r}^{1}|v^{\prime}(s)|\mathrm{d}s
\le  \left(\int_{r}^{1}s^{N-k}|v^{\prime}|^{k+1}\mathrm{d} s\right)^{\frac{1}{k+1}} \left(\int_{r}^{1}s^{-\frac{N-k}{k}}\mathrm{d} s\right)^{\frac{k}{k+1}}\\
&= \left[\overline{c}(r^{-\frac{N-2k}{k}}-1)\right]^{\frac{k}{k+1}}\|v\|_{X_1},
\end{aligned}
\end{equation}
where $\overline{c}=k(\omega_{N,k})^{-1/k}/(N-2k)$.
\end{proof}  
The next result is a generalization of the supercritical Sobolev inequality of Do \'{O}, Ruf, and Ubilla \cite[Theorem~2.1]{MR3514752} to $k$-Hessian critical exponent $k^*$, $k\ge 2$. Our proof is inspired by \cite[Lemma~2.3]{Ngu} which allows us to assume a behavior of $f$ near $0$  weaker than that of \cite{MR3514752}, and we do not assume any behavior of $f$ near $1$.
\begin{lemma}\label{super-bounded}
Suppose that $f\in C([0,1),\mathbb{R}^{+})$ satisfies the conditions
\begin{enumerate}
\item [$(f_1)$] $f(0)=0$ and $f(r)>0$ for all $r>0$,\\

\item [$(f_2)$] there exists $c>0$ such that $f(r)|\log r|\le c$,
for $r$ near $0$.
\end{enumerate}
Then
\begin{equation}\nonumber
\sup_{v\in X_1,\; \|v\|_{X_1}=1}\int_{0}^{1}r^{N-1}|v|^{k^*+f(r)}\ud r<\infty.
\end{equation}
\end{lemma}
\begin{proof}
For any $v\in X_1$, with $\|v\|_{X_1}=1$, Lemma~\ref{r-typeLemma} yields
\begin{equation}\label{radial-lemma}
|v(r)|\le \frac{c}{r^\frac{N}{k^*}}, \;\; 0<r\le 1,
\end{equation}
for some constant $c=c(k,N)>1$. We write
\begin{eqnarray}
\int_{0}^{1}r^{N-1}|v|^{k^{*}+f(r)}\ud r
&=&\int_{0}^{\rho}r^{N-1}|v|^{k^{*}+f(r)}\mathrm{d} r+\int_{\rho}^{1}r^{N-1}|v|^{k^{*}+f(r)}\mathrm{d}r, \label{twopieces}
\end{eqnarray}
where $0<\rho<1$ will be chosen later. Firstly, from \eqref{radial-lemma}
\begin{equation}\nonumber
f(r)\log\left(\frac{c}{r^\frac{N}{k^*}}\right)\le \left(\log c+\frac{N}{k^*}|\log r|\right)f(r).
\end{equation}
Thus, the assumptions $(f_1)$ and $(f_2)$ together with the continuity of $f$ imply
\begin{equation}\nonumber
\sup_{r\in (0,\rho]}\left(\frac{c}{r^\frac{N}{k^*}}\right)^{f(r)}\le c_0
\end{equation}
for some $c_0>0$. Hence, using \eqref{Ebeddings} it follows that
\begin{equation}\label{half}
\begin{aligned}
\int_{0}^{\rho}r^{N-1}|v|^{k^{*}+f(r)}\mathrm{d} r &\le \int_{0}^{\rho}r^{N-1}|v|^{k^{*}}\left(\frac{c}{r^{\frac{N}{k^*}}}\right)^{f(r)}\mathrm{d} r\\
& \le c_0\int_{0}^{1}r^{N-1}|v|^{k^{*}}\ud r\le C,\\
\end{aligned}
\end{equation}
for some $C>0$, which doesn't depend of $v\in X_1$, with $\|v\|_{X_1}=1$. Further, directly  from \eqref{2.X}, choosing $\rho=(\overline{c}/(1+\overline{c}))^{k/(N-2k)}$ one has
\begin{equation}\nonumber
\begin{aligned}
|v(r)|\le 1,\;\; \mbox{for all}\; r\ge 
\rho.
\end{aligned}
\end{equation}
Hence
\begin{equation}\label{half2}
\begin{aligned}
\int_{\rho}^{1}r^{N-1}|v|^{k^{*}+f(r)}\ud r &\le \int_{\rho}^{1}\ud r\le 1.
\end{aligned}
\end{equation}
Thus, \eqref{half} and \eqref{half2} together with \eqref{twopieces} yield our result.
\end{proof}
\subsection{Proof Corollary~\ref{coro1}} Let $u\in \Phi^{k}_{0,\mathrm{rad}}(B)$ be arbitrary nonzero function. From  Theorem~\ref{thm1}, there is $C>0$ such that
\begin{equation}\nonumber
\begin{aligned}
\int_{B}\left|\frac{u(x)}{\|u\|_{\Phi^{k}_0}}\right|^{k^{*}+|x|^{\alpha}}\ud x\le C.
\end{aligned}
\end{equation}
Thus, for any $\lambda>1$
\begin{equation}\nonumber
\begin{aligned}
\int_{B}\left|\frac{u(x)}{\lambda\|u\|_{\Phi^{k}_0}}\right|^{k^{*}+|x|^{\alpha}}\ud x
& \le  \frac{1}{\lambda^{k^*}} \int_{B}\left|\frac{u(x)}{\|u\|_{\Phi^{k}_0}}\right|^{k^{*}+|x|^{\alpha}}\ud x
\le \frac{C}{\lambda^{k^*}}.
\end{aligned}
\end{equation}
Consequently, by taking $\lambda_{*}>0$ such that $C/\lambda^{k^*}_{*}\le 1$ one has
\begin{equation}\nonumber
\|u\|_{L_{k^{*}}+|x|^{\alpha}}(B)\le \lambda_{*}\|u\|_{\Phi^{k}_0}
\end{equation}
and the proof is completed.
\section{An auxiliary extremal problem}
\label{sec3}
\noindent Let us denote by $\mathcal{V}_{k,N}$ the best constant to the embedding \eqref{Ebeddings}, for $q=k^*$ and $R=1$. Namely,
\begin{equation}\label{best-sob}
\mathcal{V}_{k,N}= \sup_{\|v\|_{X_1}=1} \int_{0}^{1}r^{N-1}|v|^{k^{*}}\mathrm{d} r.
\end{equation}
It is well known that the supremum in \eqref{best-sob} is not attained (cf. \cite{Clement-deFigueiredo-Mitidieri}), nevertheless we will be able to show that the supercritical extremal problem \eqref{supremumX} is attained. Precisely,
\begin{proposition}\label{prop1} Suppose $1\le k<N/2$ and $0<\alpha<(N-2k)/k$. Then
\begin{equation}\nonumber
\mathcal{V}_{k,N,\alpha}=\sup_{\|v\|_{X_1}=1}\int_{0}^{1}r^{N-1}|v|^{k^{*}+r^{\alpha}}\mathrm{d}r
\end{equation} 
is attained for some $v\in X_1$. 
\end{proposition}
\noindent From now on we use the following notation:
\begin{definition}
A sequence $(v_{n})\subset X_1$ is normalized concentrating sequence at the origin if 
\begin{equation}
\|v_n\|_{X_1}=1,\;\; v_n\rightharpoonup 0\;\;\;\mbox{weakly in}\;\; X_1\;\;\;\mbox{and}\;\;\int_{r_0}^{1}r^{N-k}
|v_n^{\prime}|^{k+1}\ud r\rightarrow 0,\;\; \forall\, r_0>0.
\end{equation}
\end{definition}
\noindent To prove Proposition~\ref{prop1}, it is sufficient to show the following three steps:
\\
\paragraph{\textbf{Step 1:}} The strict inequality 
$
\mathcal{V}_{k,N}<\mathcal{V}_{k,N,\alpha}
$
holds;\\
\paragraph{\textbf{Step 2:}}
If $(v_{n})\subset X_1$ is any normalized concentrating sequence at origin, then
\begin{equation}\nonumber
\limsup_{n}\int_{0}^{1}r^{N-1}|v_n|^{k^*+r^{\alpha}}\ud r\le \mathcal{V}_{k,N};
\end{equation}
\paragraph{\textbf{Step 3:}} Let $(v_{n})\subset X_1$ be any maximizing sequence for $\mathcal{V}_{k,N,\alpha}$ in $X_1$. Then, either $(v_{n})$ is normalized concentrating at origin or $\mathcal{V}_{k,N,\alpha}$ is attained.\\

\noindent The rest of this section is devoted to prove that these three steps hold. 
\subsection{Proof of Step 1}
Firstly,  for each $0<R\le\infty$,  we define
\begin{equation}\label{SSp}
S_0(k^{*},R)=\inf\left\{\int_{0}^{R}r^{N-k}|v^{\prime}|^{k+1}\ud r\;\;|\;\; v\in X_R;\;\; \int_{0}^{R}r^{N-1}|v|^{k^{*}}\ud r=1\right\}.
\end{equation}
It is known that $S_0(k^{*},R)$ is independent of $R$, and that it is achieved when $R=+\infty$ (see \cite{Clement-deFigueiredo-Mitidieri}, for more details). In addition, the functions
\begin{equation}\label{instatons}
v^{*}_{\epsilon}(r)=\hat{c}\left(\frac{\epsilon^{\frac{2}{k+1}}}{\epsilon^{2}+r^2}\right)^{\frac{N-2k}{2k}},\;\; \epsilon>0
\end{equation}
with
\begin{equation}\label{snmc}
\hat{c}=\left[N\left(\frac{N-2k}{k}\right)^{k}\right]^{\frac{N-2k}{2k}}
\end{equation}
satisfy 
\begin{equation}\label{SSpu}
S^{\frac{N}{2k}}=\int_{0}^{\infty}r^{N-k}|(v^{*}_{\epsilon})^{\prime}|^{k+1}\ud r=\int_{0}^{\infty}r^{N-1}{|v^{*}_{\epsilon}}|^{k^{*}}\ud r,
\end{equation}
where  $S$ denotes the  value of $S_0(k^{*},R)$ for $R=+\infty$, and then for any $R>0$. 
\begin{remark}\label{relation-VS} Note that $(\omega_{N,k}S)^{k^*/(k+1)}\mathcal{V}_{k,N}=1$.
\end{remark} 
\begin{lemma}\label{lemma-Isharpiii} For any $\beta,\delta\ge 0$ and $0<\gamma<1$, there holds
\begin{equation}\label{Isharpiii}
\begin{aligned}
&\int_{0}^{\epsilon^{\gamma}}r^{N+\beta-1}|v^{*}_{\epsilon}|^{k^*}\left(\log\left(1+\frac{r^2}{\epsilon^{2}}\right)\right)^{\delta}\ud r\\
&=\left\{\begin{aligned}
&\hat{c}^{k^*}\epsilon^{\beta}\int_{0}^{\infty}\frac{s^{N+\beta-1}\left(\log\left(1+s^2\right)\right)^{\delta}}{(1+s^2)^{\frac{k+1}{2}\frac{N}{k}}}\ud s+o(\epsilon^{\beta})_{\epsilon\searrow0},&\;\;\mbox{if}\;\;& \beta<N/k\\
&\hat{c}^{k^*}\frac{(2(1-\gamma))^{\delta+1}}{2(\delta+1)}\epsilon^{\frac{N}{k}}(-\log\epsilon)^{\delta+1}+o(\epsilon^{\frac{N}{k}}(-\log\epsilon)^{\delta+1})_{\epsilon\searrow 0},&\;\;\mbox{if}\;\;& \beta=N/k\\
&\hat{c}^{k^*}\frac{(2(1-\gamma))^{\delta}}{\beta-\frac{N}{k}}\epsilon^{(\beta-\frac{N}{k})\gamma+\frac{N}{k}}(-\log\epsilon)^{\delta}+o(\epsilon^{(\beta-\frac{N}{k})\gamma+\frac{N}{k}}(-\log\epsilon)^{\delta})_{\epsilon\searrow 0},&\;\;\mbox{if}\;\;& \beta>N/k.
\end{aligned}\right.
\end{aligned}
\end{equation}
\end{lemma}
\begin{proof}
Using \eqref{instatons}, by change of variables we have
\begin{equation}\label{Isharp}
\int_{0}^{\epsilon^{\gamma}}r^{N+\beta-1}|v^{*}_{\epsilon}|^{k^*}\left(\log\left(1+\frac{r^2}{\epsilon^{2}}\right)\right)^{\delta}\ud r=\hat{c}^{k^*}\epsilon^{\beta}\int_{0}^{\epsilon^{\gamma-1}}\frac{s^{N+\beta-1}}{(1+s^{2})^{\frac{k+1}{2}\frac{N}{k}}}\left(\log\left(1+s^2\right)\right)^{\delta}\ud s.
\end{equation}
If $\beta<N/k$, we have
\begin{equation}\label{Isha<}
\begin{aligned}
\int_{0}^{\infty}\frac{s^{N+\beta-1}}{(1+s^{2})^{\frac{k+1}{2}\frac{N}{k}}}\left(\log\left(1+s^2\right)\right)^{\delta}\ud s 
&\le \int_{0}^{\infty}\tau^{\delta} e^{-\frac{1}{2}(\frac{N}{k}-\beta)\tau}\ud \tau.
\end{aligned}
\end{equation}
If $\beta=N/k$, from the L'Hôpital rule we easily show that
\begin{equation}\label{Isha=}
\begin{aligned}
\lim_{\epsilon\rightarrow0}\frac{1}{(-\log\epsilon)^{\delta+1}}\int_{0}^{\epsilon^{\gamma-1}}\frac{s^{N+\beta-1}}{(1+s^2)^{\frac{k+1}{2}\frac{N}{k}}}\left(\log\left(1+s^2\right)\right)^{\delta}\ud s =\frac{(2(1-\gamma))^{\delta+1}}{2(\delta+1)}.
\end{aligned}
\end{equation}
Analogously, for $\beta>N/k$ we also get
\begin{equation}\label{Isha>}
\begin{aligned}
\lim_{\epsilon\rightarrow0}\frac{1}{\epsilon^{(\beta-\frac{N}{k})(\gamma-1)}(-\log\epsilon)^{\delta}}\int_{0}^{\epsilon^{\gamma-1}}\frac{s^{N+\beta-1}}{(1+s^2)^{\frac{k+1}{2}\frac{N}{k}}}\left(\log\left(1+s^2\right)\right)^{\delta}\ud s=\frac{(2(1-\gamma))^{\delta}}{\beta-\frac{N}{k}}.
\end{aligned}
\end{equation}
It follows from \eqref{Isharp}, \eqref{Isha<}, \eqref{Isha=}, and \eqref{Isha>} that \eqref{Isharpiii} holds.
\end{proof}
\begin{lemma}\label{lemma-Isharpiii-near1} Let $\beta\ge 0$ and $0<\gamma<1$. Then
\begin{equation}\label{Isharpiii-near1}
\begin{aligned}
&\int_{\epsilon^{\gamma}}^{1}r^{N+\beta-1}|v^{*}_{\epsilon}|^{k^*}\ud r
&=\left\{\begin{aligned}
&O\left(\epsilon^{(1-\gamma)\frac{N}{k}+\beta\gamma}\right)_{\epsilon\searrow0},\;\;&\mbox{if}&\;\; \beta<N/k\\
&O\left(\epsilon^{\frac{N}{k}}(-\log\epsilon)\right)_{\epsilon\searrow0},\;\;&\mbox{if}&\;\; \beta=N/k\\
& O\left(\epsilon^{\frac{N}{k}}\right)_{\epsilon\searrow 0},\;\;&\mbox{if}&\;\; \beta>N/k.\\
\end{aligned}\right.
\end{aligned}
\end{equation}
\end{lemma}
\begin{proof}
By change of variables we can write
\begin{equation}\nonumber
\begin{aligned}
\int_{\epsilon^{\gamma}}^{1}r^{N+\beta-1}|v^{*}_{\epsilon}|^{k^*}\ud r& =\hat{c}^{k^*}\epsilon^{\beta}\int_{\epsilon^{\gamma-1}}^{\epsilon^{-1}}\frac{s^{\beta-\frac{N}{k}-1}}{\left(1+s^{-2}\right)^{\frac{k+1}{2}\frac{N}{k}}}\ud s.
\end{aligned}
\end{equation}
If $\beta<N/k$, one has
\begin{equation}\nonumber
\begin{aligned}
\lim_{\epsilon\rightarrow 0}\frac{1}{\epsilon^{(1-\gamma)(\frac{N}{k}-\beta)}}\int_{\epsilon^{\gamma-1}}^{\epsilon^{-1}}\frac{s^{\beta-\frac{N}{k}-1}}{\left(1+s^{-2}\right)^{\frac{k+1}{2}\frac{N}{k}}}\ud s
=\frac{1}{(\frac{N}{k}-\beta)}.
\end{aligned}
\end{equation}
If $\beta=N/k$
\begin{equation}\nonumber
\begin{aligned}
0\le \lim_{\epsilon\rightarrow 0}\frac{1}{-\log\epsilon}\int_{\epsilon^{\gamma-1}}^{\epsilon^{-1}}\frac{s^{\beta-\frac{N}{k}-1}}{\left(1+s^{-2}\right)^{\frac{k+1}{2}\frac{N}{k}}}\ud s&=\lim_{\epsilon\rightarrow 0}\frac{1}{-\log\epsilon}\int_{\epsilon^{\gamma-1}}^{\epsilon^{-1}}\frac{1}{\left(1+s^{-2}\right)^{\frac{k+1}{2}\frac{N}{k}}}\frac{1}{s}\ud s\\
&\le \lim_{\epsilon\rightarrow 0}\frac{1}{-\log\epsilon}\int_{\epsilon^{\gamma-1}}^{\epsilon^{-1}}\frac{1}{s}\ud s=\gamma.
\end{aligned}
\end{equation}
Finally, for $\beta>N/k$
\begin{equation}\nonumber
\begin{aligned}
\lim_{\epsilon\rightarrow 0}\epsilon^{(\beta-\frac{N}{k})}\int_{\epsilon^{\gamma-1}}^{\epsilon^{-1}}\frac{s^{\beta-\frac{N}{k}-1}}{\left(1+s^{-2}\right)^{\frac{k+1}{2}\frac{N}{k}}}\ud s& \le \lim_{\epsilon\rightarrow 0}\epsilon^{(\beta-\frac{N}{k})}\int_{\epsilon^{\gamma-1}}^{\epsilon^{-1}}s^{\beta-\frac{N}{k}-1}\ud s\\
&=\frac{1}{(\beta-\frac{N}{k})}.
\end{aligned}
\end{equation}
\end{proof}
\begin{lemma} \label{lemma-Isharpiii-rabo} Let $0<\gamma<1$, then
\begin{equation}\nonumber
\int_{1}^{\infty}r^{N-1}|v^{*}_{\epsilon}|^{k^*}\ud r=o\left(\epsilon^{\frac{N}{k}(1-\gamma)}\right)_{\epsilon\searrow 0}.
\end{equation}
\end{lemma}
\begin{proof} Indeed, we have
\begin{equation}\nonumber
\begin{aligned}
\int_{1}^{\infty}r^{N-1}|v^{*}_{\epsilon}|^{k^*}\ud r& =\hat{c}^{k^*}\int_{\epsilon^{-1}}^{\infty}\frac{s^{-\frac{N}{k}-1}}{\left(1+s^{-2}\right)^{\frac{k+1}{2}\frac{N}{k}}}\ud s,
\end{aligned}
\end{equation}
and
\begin{equation}\nonumber
\begin{aligned}
\lim_{\epsilon\rightarrow 0}\frac{1}{\epsilon^{(1-\gamma)\frac{N}{k}}}\int_{\epsilon^{-1}}^{\infty}\frac{s^{-\frac{N}{k}-1}}{\left(1+s^{-2}\right)^{\frac{k+1}{2}\frac{N}{k}}}\ud s
=\frac{1}{(1-\gamma)\frac{N}{k}}\lim_{\epsilon\rightarrow 0}\frac{1}{\epsilon^{(1-\gamma)\frac{N}{k}-1}}\frac{\epsilon^{\frac{N}{k}-1}}{(1+\epsilon^2)^{\frac{N}{k}\frac{k+1}{2}}}=0.
\end{aligned}
\end{equation}
\end{proof}
\noindent Let us consider $\eta\in C^{\infty}_{0}(0,1)$ to be a fixed cut-off function satisfying
\begin{equation}\label{cut-off}
 0\le\eta\le 1,\;\; \eta(r)\equiv 1,\;\;\forall\; r\in (0,r_0] \;\;\mbox{and}\quad \eta(r)\equiv 0, \;\; \forall\; r \in [2r_0,1],
\end{equation}
for some $0<r_0<2r_0<1$. 

\noindent In the same line of \cite{MR0709644}, the following result was recently shown in \cite{JDE2019}:
\begin{lemma}\label{lemmaJDE} The family $(v^{*}_{\epsilon})_{\epsilon>0}$ given by \eqref{instatons}  satisfies:  
 
\begin{enumerate}
\item[$\mathrm{(a)}$] $\int_{0}^{1}r^{N-k}|(\eta v^{*}_{\epsilon})^{\prime}|^{k+1}\ud r = S^{\frac{N}{2k}}+O(\epsilon^{\frac{N-2k}{k}})_{\epsilon\searrow 0}
$\\
\item[$\mathrm{(b)}$] $\int_{0}^{1}r^{N-1}|\eta v^{*}_{\epsilon}|^{k^{*}}\ud r=S^{\frac{N}{2k}}+O(\epsilon^{\frac{N}{k}})_{\epsilon\searrow 0}$ \\
\end{enumerate}
where $\eta$ is the cut-off function given by \eqref{cut-off}.
\end{lemma}
\noindent From now on we will denote
\begin{equation}\label{we}
w_{\epsilon}(r)=\mathcal{C}\eta(r) v^{*}_{\epsilon}(r)=A\eta(r)\left(\frac{\epsilon^{\frac{2}{k+1}}}{\epsilon^{2}+r^2}\right)^{\frac{N-2k}{2k}},\;\;\epsilon>0
\end{equation}
where $\mathcal{C}>0$ is arbitrary and $A=\mathcal{C}\hat{c}$.
\begin{lemma} \label{lemma-sharper}Let $\alpha>0$ and $(w_{\epsilon})$ be given in \eqref{we}. Then, there exists $\mathcal{C}_{1}>0$ such that
\begin{equation}\label{lemma-estimate}
\int_{0}^{1}r^{N-1}|w_{\epsilon}|^{k^*+r^{\alpha}}\ud r = \mathcal{C}^{k^*}S^{\frac{N}{2k}}+ \left\{
\begin{aligned}
& \mathcal{C}_{1}|\log\epsilon|\epsilon^{\alpha}+o\left(\epsilon^{\alpha}|\log\epsilon|\right)_{{\epsilon\searrow0}},\;\;&\mbox{if}&\;\;\alpha<N/k\\
& O\left(\epsilon^{\frac{N}{k}(1-\gamma)}\right)_{\epsilon\searrow 0},\;\;&\mbox{if}&\;\; \alpha\ge N/k
\end{aligned}\right.
\end{equation}
and 
\begin{equation}\label{lemma-estimate2-sharp}
\int_{0}^{1}\frac{r^{N-1}}{k^*+r^{\alpha}}|w_{\epsilon}|^{k^{*}+r^{\alpha}}\ud r = \frac{\mathcal{C}^{k^*}S^{\frac{N}{2k}}}{k^*}+ \left\{
\begin{aligned}
& \frac{\mathcal{C}_{1}}{k^*}|\log\epsilon|\epsilon^{\alpha}+o\left(\epsilon^{\alpha}|\log\epsilon|\right)_{\epsilon\searrow0},\;\;&\mbox{if}&\;\;\alpha<N/k\\
& O\left(\epsilon^{\frac{N}{k}(1-\gamma)}\right)_{\epsilon\searrow 0},\;\;&\mbox{if}&\;\; \alpha\ge N/k
\end{aligned}\right.
\end{equation} 
for arbitrary $0<\gamma<1/(k+1)$ small enough. 
\end{lemma}
\begin{proof}
It follows from the definition of  $v^{*}_{\epsilon}$ in \eqref{instatons} that $\mathcal{C}v^{*}_{\epsilon}(r)\le 1$ if and only if
\begin{equation}\label{a-eps}
r\ge \epsilon^{\frac{1}{k+1}}\sqrt{A^{\frac{2k}{N-2k}}-\epsilon^{\frac{2k}{k+1}}}:=a_{\epsilon}.
\end{equation}
In particular, for any $0<\gamma<1/(k+1)$, we have
\begin{equation}\label{aer_0}
0<\epsilon< a_{\epsilon}<
\epsilon^{\gamma}<r_0<1,
\end{equation}
if $\epsilon>0$ is sufficiently small. Since $w_{\epsilon}=\mathcal{C}\eta v^{*}_{\epsilon}\le 1$ in $(\epsilon^{\gamma}, 1)$, it follows that 
\begin{equation}\nonumber
\begin{aligned}
\int_{\epsilon^{\gamma}}^{1}r^{N-1}|w_{\epsilon}|^{k^{*}+r^{\alpha}}\ud r & \le  \mathcal{C}^{k^*}\int_{\epsilon^{\gamma}}^{1}r^{N-1}|v^{*}_{\epsilon}|^{k^{*}}\ud r.
\end{aligned}
\end{equation}
Hence, Lemma~\ref{lemma-Isharpiii-near1} with $\beta=0$ yields
\begin{equation}\label{exp-var-Lq}
\int_{\epsilon^{\gamma}}^{1}r^{N-1}|w_{\epsilon}|^{k^{*}+r^{\alpha}}\ud r=O\left(\epsilon^{\frac{N}{k}(1-\gamma)}\right)_{\epsilon\searrow 0}.
\end{equation}
Analogously,  we can see that
\begin{equation}\label{exp-Lq}
\int_{\epsilon^{\gamma}}^{1}r^{N-1}|\mathcal{C}v^{*}_{\epsilon}|^{k^{*}}\ud r=O\left(\epsilon^{\frac{N}{k}(1-\gamma)}\right)_{\epsilon\searrow 0}.
\end{equation}
Therefore, this together with \eqref{SSpu} and Lemma~\ref{lemma-Isharpiii-rabo} provides
\begin{equation}\label{sharp1}
\begin{aligned}
\int_{0}^{1}r^{N-1}|w_{\epsilon}|^{k^*}\ud r&=\mathcal{C}^{k^*}\int_{0}^{1}r^{N-1}|v^{*}_{\epsilon}|^{k^*}\ud r+\int_{0}^{1}r^{N-1}(\eta^{k^*}-1)|\mathcal{C}v^{*}_{\epsilon}|^{k^*}\ud r\\
&=\mathcal{C}^{k^*}S^{\frac{N}{2k}}+O\left(\epsilon^{\frac{N}{k}(1-\gamma)}\right)_{\epsilon\searrow 0},
\end{aligned}
\end{equation}
because $\eta\equiv 1$ on $(0,\epsilon^{\gamma})$.
It also follows from \eqref{exp-Lq} and \eqref{sharp1} that
\begin{equation}\label{sharp4}
\int_{0}^{\epsilon^{\gamma}}r^{N-1}|w_{\epsilon}|^{k^*}\ud r=\mathcal{C}^{k^*}S^{\frac{N}{2k}}+O\left(\epsilon^{\frac{N}{k}(1-\gamma)}\right)_{\epsilon\searrow 0}.
\end{equation}
On the other hand, we have $
\mathcal{C}\eta v^{*}_{\epsilon}\le 1$ and $\eta\equiv 1 $ in $(a_{\epsilon}, \epsilon^ {\gamma})$. Then, for any $r\in (a_{\epsilon}, \epsilon^ {\gamma}) $
\begin{equation}\nonumber
\begin{aligned}
1\ge  w_{\epsilon}(r)&\ge A\left(\frac{\epsilon^{\frac{2}{k+1}}}{\epsilon^{2}+\epsilon^{2\gamma}}\right)^{\frac{N-2k}{2k}}=A\epsilon^{\frac{N-2k}{k}\left(\frac{1}{k+1}-\gamma\right)}(1+\epsilon^{2(1-\gamma)})^{-\frac{N-2k}{2k}}
\end{aligned}
\end{equation}
which implies
\begin{equation}\nonumber
\begin{aligned}
0& \ge r^{\alpha}\log |w_{\epsilon}(r)| \ge \epsilon^{\alpha\gamma}\log|w_{\epsilon}(r)|\\
&\ge \frac{N-2k}{k}\left(\frac{1}{k+1}-\gamma\right)\epsilon^{\alpha\gamma}\log\epsilon+\epsilon^{\alpha\gamma}\log\left(A(1+\epsilon^{2(1-\gamma)})^{-\frac{N-2k}{2k}}\right)\\
&=o(1)_{\epsilon\searrow 0},
\end{aligned}
\end{equation}
for any $r\in (a_{\epsilon}, \epsilon^ {\gamma})$.
Hence, by using Taylor's expansion we can write
\begin{equation}\label{Taylor1}
\begin{aligned}
|w_{\epsilon}(r)|^{r^{\alpha}}&=1+\left(\log A-\frac{N-2k}{k+1}\log\epsilon-\frac{N-2k}{2k}\log\left(1+\frac{r^2}{\epsilon^{2}}\right)\right)r^{\alpha}\\
&\;\quad+O\left(\left(\log A-\frac{N-2k}{k+1}\log\epsilon-\frac{N-2k}{2k}\log\left(1+\frac{r^2}{\epsilon^{2}}\right)\right)^{2}r^{2\alpha}\right)_{\epsilon\searrow 0},
\end{aligned} 
\end{equation}
for all $r\in(a_{\epsilon}, \epsilon^{\gamma})$. Thus we have
\begin{equation}\label{sharp2}
\begin{aligned}
&\int_{a_{\epsilon}}^{\epsilon^{\gamma}}r^{N-1}|w_{\epsilon}|^{k^*+r^{\alpha}}\ud r\!=\left(\log A-\frac{N-2k}{k+1}\log\epsilon\right)\int_{a_{\epsilon}}^{\epsilon^{\gamma}}r^{N+\alpha-1}|w_{\epsilon}|^{k^*}\ud r\\
&\;+\int_{a_{\epsilon}}^{\epsilon^{\gamma}}r^{N-1}|w_{\epsilon}|^{k^*}\ud r-\frac{N-2k}{2k}\int_{a_{\epsilon}}^{\epsilon^{\gamma}}r^{N+\alpha-1}\log\left(1+\frac{r^2}{\epsilon^{2}}\right)|w_{\epsilon}|^{k^*}\ud r\\
&\;\;\;+O\left(\int_{a_{\epsilon}}^{\epsilon^{\gamma}}r^{N-1}|w_{\epsilon}|^{k^*}\left(\log A-\frac{N-2k}{k+1}\log\epsilon-\frac{N-2k}{2k}\log\left(1+\frac{r^2}{\epsilon^{2}}\right)\right)^{2}r^{2\alpha}\ud r\right)_{\epsilon\searrow 0}.
\end{aligned}
\end{equation}
Now, our choice \eqref{aer_0} gives $w_{\epsilon}=\mathcal{C}v^{*}_{\epsilon}\ge 1$ on $(0,a_{\epsilon})$ then
\begin{equation}
\begin{aligned}
0 &\le \left(\log A-\frac{N-2k}{k+1}\log \epsilon-\frac{N-2k}{2k}\log\left(1+\frac{r^2}{\epsilon^2}\right)\right)r^{\alpha}\\
& \le \left(\log A-\frac{N-2k}{k+1}\log \epsilon\right)a_{\epsilon}^{\alpha}\\
&=O\left(\epsilon^{\frac{\alpha}{k+1}}(-\log\epsilon)\right)_{\epsilon\searrow 0}=o(1)_{\epsilon\searrow 0},
\end{aligned}
\end{equation}
for any $r\in (0,a_{\epsilon})$. Hence, from Taylor’s expansion again we obtain
\begin{equation}\label{Taylor2}
\begin{aligned}
|w_{\epsilon}|^{r^{\alpha}}&=1+\left(\log A-\frac{N-2k}{k+1}\log \epsilon-\frac{N-2k}{2k}\log\left(1+\frac{r^2}{\epsilon^2}\right)\right)r^{\alpha}\\
&+O\left(\left(\log A-\frac{N-2k}{k+1}\log \epsilon-\frac{N-2k}{2k}\log\left(1+\frac{r^2}{\epsilon^2}\right)\right)^{2}r^{2\alpha}\right)_{\epsilon\searrow 0}.
\end{aligned}
\end{equation}
Therefore,
\begin{equation}\label{sharp3}
\begin{aligned}
&\int_{0}^{a_{\epsilon}}r^{N-1}|w_{\epsilon}|^{k^*+r^{\alpha}}\ud r\!=\left(\log A-\frac{N-2k}{k+1}\log\epsilon\right)\int_{0}^{a_{\epsilon}}r^{N+\alpha-1}|w_{\epsilon}|^{k^*}\ud r\\
&\;+\int_{0}^{a_{\epsilon}}r^{N-1}|w_{\epsilon}|^{k^*}\ud r-\frac{N-2k}{2k}\int_{0}^{a_{\epsilon}}r^{N+\alpha-1}\log\left(1+\frac{r^2}{\epsilon^{2}}\right)|w_{\epsilon}|^{k^*}\ud r\\
&\;\;\;+O\left(\int_{0}^{a_{\epsilon}}r^{N-1}|w_{\epsilon}|^{k^*}\left(\log A-\frac{N-2k}{k+1}\log\epsilon-\frac{N-2k}{2k}\log\left(1+\frac{r^2}{\epsilon^{2}}\right)\right)^{2}r^{2\alpha}\ud r\right)_{\epsilon\searrow 0}.
\end{aligned}
\end{equation}
It follows from \eqref{exp-var-Lq}, \eqref{sharp4} \eqref{sharp2}, and \eqref{sharp3} that
\begin{equation}\label{sharp5}
\begin{aligned}
&\int_{0}^{1}r^{N-1}|w_{\epsilon}|^{k^{*}+r^{\alpha}}\ud r 
=\mathcal{C}^{k^*}S^{\frac{N}{2k}}+O\left(\epsilon^{\frac{N}{k}(1-\gamma)}\right)_{\epsilon\searrow 0}\\
&+\left(\log A-\frac{N-2k}{k+1}\log\epsilon\right)\int_{0}^{\epsilon^{\gamma}}r^{N+\alpha-1}|w_{\epsilon}|^{k^*}\ud r\\
&\;-\frac{N-2k}{2k}\int_{0}^{\epsilon^{\gamma}}r^{N+\alpha-1}|w_{\epsilon}|^{k^*}\log\left(1+\frac{r^2}{\epsilon^{2}}\right)\ud r\\
&\;\;\;+O\left(\int_{0}^{\epsilon^{\gamma}}r^{N+2\alpha-1}|w_{\epsilon}|^{k^*}\left(\log A-\frac{N-2k}{k+1}\log\epsilon-\frac{N-2k}{2k}\log\left(1+\frac{r^2}{\epsilon^{2}}\right)\right)^{2}\ud r\right)_{\epsilon\searrow 0}.
\end{aligned}
\end{equation}
Next, we will estimate all integrals on the right hand side of \eqref{sharp5}.

From Lemma~\ref{lemma-Isharpiii} with for $\delta=0$ and $\alpha=\beta$, we have
\begin{equation}\label{Isharpiii-part1}
\begin{aligned}
&\left(\log A-\frac{N-2k}{k+1}\log\epsilon\right)\int_{0}^{\epsilon^{\gamma}}r^{N+\alpha-1}|w_{\epsilon}|^{k^*}\ud r\\
&=\left\{\begin{aligned}
&\frac{N-2k}{k+1}A^{k^*}\left(\epsilon^{\alpha}|\log\epsilon|\right)\int_{0}^{\infty}\frac{s^{N+\alpha-1}}{(1+s^2)^{\frac{k+1}{2}\frac{N}{k}}}\ud s+o\left(\epsilon^{\alpha}|\log\epsilon|\right)_{\epsilon\searrow 0},&\;\;\mbox{if}\;\;& \alpha<N/k\\
& \frac{N-2k}{k+1}A^{k^*}(1-\gamma)\epsilon^{\frac{N}{k}}(\log\epsilon)^{2}+o\left(\epsilon^{\frac{N}{k}}(\log\epsilon)^{2}\right)_{\epsilon\searrow 0},&\;\;\mbox{if}\;\;& \alpha=N/k\\
&\frac{N-2k}{k+1}A^{k^*}\frac{1}{(\alpha-\frac{N}{k})}\epsilon^{(\alpha-\frac{N}{k})\gamma+\frac{N}{k}}|\log\epsilon|+o\left(\epsilon^{(\alpha-\frac{N}{k})\gamma+\frac{N}{k}}|\log\epsilon|\right)_{\epsilon\searrow 0},&\;\;\mbox{if}\;\;& \alpha>N/k.
\end{aligned}\right.
\end{aligned}
\end{equation}
Similarly, we obtain
\begin{equation}\label{Isharpiii-part2}
\begin{aligned}
&\int_{0}^{\epsilon^{\gamma}}r^{N+\alpha-1}|w_{\epsilon}|^{k^*}\log\left(1+\frac{r^2}{\epsilon^{2}}\right)\ud r\\
&=\left\{\begin{aligned}
&A^{k^*}\epsilon^{\alpha}\int_{0}^{\infty}\frac{s^{N+\alpha-1}\log\left(1+s^2\right)}{(1+s^2)^{\frac{k+1}{2}\frac{N}{k}}}\ud s+o(\epsilon^{\alpha})_{\epsilon\searrow 0},&\;\;\mbox{if}\;\;& \alpha<N/k\\
&A^{k^*}(1-\gamma)^{2}\epsilon^{\frac{N}{k}}(\log\epsilon)^{2}+o(\epsilon^{\frac{N}{k}}(\log\epsilon)^{2})_{\epsilon\searrow 0},&\;\;\mbox{if}\;\;& \alpha=N/k\\
&A^{k^*}\frac{2(1-\gamma)}{\alpha-\frac{N}{k}}\epsilon^{(\alpha-\frac{N}{k})\gamma+\frac{N}{k}}|\log\epsilon|+o(\epsilon^{(\alpha-\frac{N}{k})\gamma+\frac{N}{k}}|\log\epsilon|)_{\epsilon\searrow 0},&\;\;\mbox{if}\;\;& \alpha>N/k.
\end{aligned}\right.
\end{aligned}
\end{equation}
Finally, we also have
\begin{equation}\label{Isharpiii-part3}
\begin{aligned}
&\int_{0}^{\epsilon^{\gamma}}r^{N+2\alpha-1}|w_{\epsilon}|^{k^*}\left(\log A-\frac{N-2k}{k+1}\log\epsilon-\frac{N-2k}{2k}\log\left(1+\frac{r^2}{\epsilon^{2}}\right)\right)^{2}\ud r\\
&=\left\{\begin{aligned}
&O\left(\epsilon^{2\alpha}(\log\epsilon)^2\right)_{\epsilon\searrow 0},&\;\;\mbox{if}\;\;& 2\alpha<N/k\\
&O\left(\epsilon^{\frac{N}{k}}|\log\epsilon|^3\right)_{\epsilon\searrow 0},&\;\;\mbox{if}\;\;& 2\alpha=N/k\\
&O\left(\epsilon^{(2\alpha-\frac{N}{k})\gamma+\frac{N}{k}}(\log\epsilon)^{2}\right)_{\epsilon\searrow 0},&\;\;\mbox{if}\;\;& 2\alpha>N/k.
\end{aligned}\right.
\end{aligned}
\end{equation}
From \eqref{sharp5}, \eqref{Isharpiii-part1}, \eqref{Isharpiii-part2} and  \eqref{Isharpiii-part3}, we have
\begin{equation}\nonumber
\begin{aligned}
\int_{0}^{1}r^{N-1}|w_{\epsilon}|^{k^{*}+r^{\alpha}}\ud r 
&=\mathcal{C}^{k^*}S^{\frac{N}{2k}}+\frac{N-2k}{k+1}A^{k^*}\left(\epsilon^{\alpha}|\log\epsilon|\right)\int_{0}^{\infty}\frac{s^{N+\alpha-1}}{(1+s^2)^{\frac{k+1}{2}\frac{N}{k}}}\ud s\\
&+o(\epsilon^{\alpha}|\log\epsilon|)_{\epsilon\searrow 0}+O\left(\epsilon^{\frac{N}{k}(1-\gamma)}\right)_{\epsilon\searrow 0},
\end{aligned}
\end{equation}
if $\alpha<N/k$. Choosing $\gamma>0$ small enough such that $N(1-\gamma)/k>\alpha$, we obtain \eqref{lemma-estimate} with
$$
\mathcal{C}_{1}=\frac{N-2k}{k+1}A^{k^*}\int_{0}^{\infty}\frac{s^{N+\alpha-1}}{(1+s^2)^{\frac{k+1}{2}\frac{N}{k}}}\ud s
$$
for $\alpha<N/k$. It also follows from \eqref{sharp5}, \eqref{Isharpiii-part1}, \eqref{Isharpiii-part2}, and  \eqref{Isharpiii-part3} that
\begin{equation}\nonumber
\begin{aligned}
\int_{0}^{1}r^{N-1}|w_{\epsilon}|^{k^{*}+r^{\alpha}}\ud r 
&=\mathcal{C}^{k^*}S^{\frac{N}{2k}}+O\left(\epsilon^{\frac{N}{k}(1-\gamma)}\right)_{\epsilon\searrow 0}+O\left(\epsilon^{\frac{N}{k}}(\log\epsilon)^2\right)_{\epsilon\searrow 0}\\
&\;\;+O\left(\epsilon^{(2\alpha-\frac{N}{k})\gamma+\frac{N}{k}}(\log\epsilon)^{2}\right)_{\epsilon\searrow 0}\\
&=\mathcal{C}^{k^*}S^{\frac{N}{2k}}+O\left(\epsilon^{\frac{N}{k}(1-\gamma)}\right)_{\epsilon\searrow 0},
\end{aligned}
\end{equation}
if $\alpha=N/k$. We also have
\begin{equation}\nonumber
\begin{aligned}
\int_{0}^{1}r^{N-1}|w_{\epsilon}|^{k^{*}+r^{\alpha}}\ud r 
&=\mathcal{C}^{k^*}S^{\frac{N}{2k}}+O\left(\epsilon^{\frac{N}{k}(1-\gamma)}\right)_{\epsilon\searrow 0}+O\left(\epsilon^{(\alpha-\frac{N}{k})\gamma+\frac{N}{k}}(-\log\epsilon)\right)_{\epsilon\searrow 0}\\
&\;\;+O\left(\epsilon^{(2\alpha-\frac{N}{k})\gamma+\frac{N}{k}}(\log\epsilon)^{2}\right)_{\epsilon\searrow 0}\\
&=\mathcal{C}^{k^*}S^{\frac{N}{2k}}+O\left(\epsilon^{\frac{N}{k}(1-\gamma)}\right)_{\epsilon\searrow 0},
\end{aligned}
\end{equation}
if $\alpha>N/k$. This proves \eqref{lemma-estimate}.

Now we proceed to the proof of the estimate \eqref{lemma-estimate2-sharp}. Firstly, 
\begin{equation}\label{s-s-first}
\begin{aligned}
\int_{0}^{1}\frac{r^{N-1}}{k^*+r^{\alpha}}|w_{\epsilon}|^{k^{*}+r^{\alpha}}\ud r &=\frac{1}{k^*}\int_{0}^{1}r^{N-1}|w_{\epsilon}|^{k^{*}+r^{\alpha}}\ud r \\
& \quad -\frac{1}{k^*}\int_{0}^{1}\frac{r^{\alpha}}{k^*+r^{\alpha}}r^{N-1}|w_{\epsilon}|^{k^{*}+r^{\alpha}}\ud r.
\end{aligned}
\end{equation}
Since $w_{\epsilon}=\mathcal{C}\eta v^{*}_{\epsilon}\le 1$ in $(\epsilon^{\gamma}, 1)$, we can write
\begin{equation}\nonumber
\begin{aligned}
0\le \int_{\epsilon^{\gamma}}^{1}\frac{r^{\alpha}}{k^*+r^{\alpha}}r^{N-1}|w_{\epsilon}|^{k^{*}+r^{\alpha}}\ud r
\le \frac{\mathcal{C}^{k^*}}{k^*} \int_{\epsilon^{\gamma}}^{1}r^{\alpha+N-1}|v^{*}_{\epsilon}|^{k^*}\ud r.
\end{aligned}
\end{equation}
Thus, Lemma~\ref{lemma-Isharpiii-near1} yields
\begin{equation}\label{sharp-split1}
\begin{aligned}
\int_{\epsilon^{\gamma}}^{1}\frac{r^{\alpha}}{k^*+r^{\alpha}}r^{N-1}|w_{\epsilon}|^{k^{*}+r^{\alpha}}\ud r
=\left\{\begin{aligned}
&O\left(\epsilon^{(1-\gamma)\frac{N}{k}+\alpha\gamma}\right)_{\epsilon\searrow 0},\;\;&\mbox{if}&\;\; \alpha<N/k\\
&O\left(\epsilon^{\frac{N}{k}}(-\log\epsilon)\right)_{\epsilon\searrow 0},\;\;&\mbox{if}&\;\; \alpha=N/k\\
& O\left(\epsilon^{\frac{N}{k}}\right)_{\epsilon\searrow 0},\;\;&\mbox{if}&\;\; \alpha>N/k.\\
\end{aligned}\right.
\end{aligned}
\end{equation}
Similarly to \eqref{sharp2} and \eqref{sharp3},   by using \eqref{Taylor1} and \eqref{Taylor2}, we can show that
\begin{equation}\label{IIsharp5}
\begin{aligned}
&\int_{0}^{\epsilon^{\gamma}}r^{\alpha+N-1}|w_{\epsilon}|^{k^{*}+r^{\alpha}}\ud r =\left(\log A-\frac{N-2k}{k+1}\log\epsilon\right)\int_{0}^{\epsilon^{\gamma}}r^{N+2\alpha-1}|w_{\epsilon}|^{k^*}\ud r\\
&\;+\int_{0}^{\epsilon^{\gamma}}r^{\alpha+N-1}|w_{\epsilon}|^{k^*}\ud r-\frac{N-2k}{2k}\int_{0}^{\epsilon^{\gamma}}r^{N+2\alpha-1}|w_{\epsilon}|^{k^*}\log\left(1+\frac{r^2}{\epsilon^{2}}\right)\ud r\\
&\;\;+O\left(\int_{0}^{\epsilon^{\gamma}}r^{N+3\alpha-1}|w_{\epsilon}|^{k^*}\left(\log A-\frac{N-2k}{k+1}\log\epsilon-\frac{N-2k}{2k}\log\left(1+\frac{r^2}{\epsilon^{2}}\right)\right)^{2}\ud r\right)_{\epsilon\searrow 0}.
\end{aligned}
\end{equation}
Finally, combining \eqref{IIsharp5} with \eqref{Isharpiii}, we can write
\begin{equation}\label{sharp-split2}
\begin{aligned}
0\le \int_{0}^{\epsilon^{\gamma}}\frac{r^{\alpha}}{k^*+r^{\alpha}}r^{N-1}|w_{\epsilon}|^{k^{*}+r^{\alpha}}\ud r&\le \frac{1}{k^*}\int_{0}^{\epsilon^{\gamma}}r^{\alpha+N-1}|w_{\epsilon}|^{k^{*}+r^{\alpha}}\ud r \\
&=\left\{\begin{aligned}
&O\left(\epsilon^{\alpha}\right)_{\epsilon\searrow 0},\;\;&\mbox{if}&\;\; \alpha<N/k\\
&O\left(\epsilon^{\alpha}(-\log\epsilon)\right)_{\epsilon\searrow 0},\;\;&\mbox{if}&\;\; \alpha=N/k\\
& O\left(\epsilon^{\frac{N}{k}(1-\gamma)+\alpha\gamma}\right)_{\epsilon\searrow 0},\;\;&\mbox{if}&\;\; \alpha>N/k.\\
\end{aligned}\right.
\end{aligned}
\end{equation}
It follows from \eqref{sharp-split1} and \eqref{sharp-split2} that
\begin{equation}\label{s-s-final}
\begin{aligned}
\int_{0}^{1}\frac{r^{\alpha}}{k^*+r^{\alpha}}r^{N-1}|w_{\epsilon}|^{k^{*}+r^{\alpha}}\ud r &=\left\{\begin{aligned}
&O\left(\epsilon^{\alpha}\right)_{\epsilon\searrow 0},\;\;&\mbox{if}&\;\; \alpha<N/k\\
&O\left(\epsilon^{\alpha}(-\log\epsilon)\right)_{\epsilon\searrow 0},\;\;&\mbox{if}&\;\; \alpha=N/k\\
& O\left(\epsilon^{\frac{N}{k}}\right)_{\epsilon\searrow 0},\;\;&\mbox{if}&\;\; \alpha>N/k.\\
\end{aligned}\right.
\end{aligned}
\end{equation}
Hence, using \eqref{lemma-estimate}, \eqref{s-s-first} and \eqref{s-s-final}, we can see that \eqref{lemma-estimate2-sharp} holds.
\end{proof}
In the next result we provide the expansion of  $\|w_{\epsilon}\|^{-(k^*+r^{\alpha})}_{X_1}$ in terms of $\epsilon$, for a suitable choice of $\mathcal{C}>0$.
\begin{lemma} If $\mathcal{C}=\left(\omega_{N,k}S^{\frac{N}{2k}}\right)^{-\frac{1}{k+1}}$,  we have
\begin{equation}\label{lemma-estimate2}
\|w_{\epsilon}\|^{-(k^*+r^{\alpha})}_{X_1}=1+O\left(\epsilon^{\frac{N-2k}{k}}\right)_{\epsilon\searrow 0},\;\;\mbox{as}\;\; \epsilon\rightarrow 0
\end{equation}
for all $r\in (0,1)$.
\end{lemma}
\begin{proof}
By Lemma~\ref{lemmaJDE}, it follows that 
\begin{equation}\label{we-grad}
\|w_{\epsilon}\|_{X_1}=1+O(\epsilon^{\frac{N-2k}{k}})_{\epsilon\searrow 0}.
\end{equation}
Hence, there exists $\delta>0$ such that
\begin{equation}\nonumber
0<1-\delta\epsilon^{\frac{N-2k}{k}}\le \|w_{\epsilon}\|_{X_1}\le 1+\delta\epsilon^{\frac{N-2k}{k}}
\end{equation}
for small enough $\epsilon>0$. Hence, for $r\in (0,1)$ there holds 
$$
\|w_{\epsilon}\|^{k^*+r^{\alpha}}_{X_1}\le 
(1+\delta\epsilon^{\frac{N-2k}{k}})^{k^*+r^{\alpha}}\le (1+\delta\epsilon^{\frac{N-2k}{k}})^{k^*+1}\le 1+\delta_{1}\epsilon^{\frac{N-2k}{k}},
$$
for some constant $\delta_1>0$. Analogously, 
$$
\|w_{\epsilon}\|^{k^*+r^{\alpha}}_{X_1}\ge 
(1-\delta\epsilon^{\frac{N-2k}{k}})^{k^*+r^{\alpha}}\ge (1-\delta\epsilon^{\frac{N-2k}{k}})^{k^*+1}\ge 1-\delta_{2}\epsilon^{\frac{N-2k}{k}},
$$
for some constant $\delta_{2}>0$. Consequently,
$$
\|w_{\epsilon}\|^{k^*+r^{\alpha}}_{X_1}=1+O\left(\epsilon^{\frac{N-2k}{k}}\right)_{\epsilon\searrow 0}
$$
and 
$$
\|w_{\epsilon}\|^{-(k^*+r^{\alpha})}_{X_1}=1+O\left(\epsilon^{\frac{N-2k}{k}}\right)_{\epsilon\searrow 0}.
$$
\end{proof}
We are now in a position to complete the proof of  Step~1. Let $(w_{\epsilon})$ be given by \eqref{we} with the choice $\mathcal{C}=\left(\omega_{N,k}S^{\frac{N}{2k}}\right)^{-\frac{1}{k+1}}$. From \eqref{lemma-estimate} and \eqref{lemma-estimate2} together with Remark~\ref{relation-VS}, we obtain
\begin{equation}\nonumber
\begin{aligned}
\mathcal{V}_{k,N,\alpha}& = \sup_{\|v\|_{X_1}=1}\int_{0}^{1}r^{N-1}|v|^{k^{*}+r^{\alpha}}\mathrm{d}r\\
& \ge \int_{0}^{1}r^{N-1}\left|\frac{w_{\epsilon}}{\|w_{\epsilon}\|_{X_1}}\right|^{k^{*}+r^{\alpha}}\mathrm{d}r\\
&= \int_{0}^{1}r^{N-1}|w_{\epsilon}|^{k^*+r^{\alpha}}\ud r + O\left(\epsilon^{\frac{N-2k}{k}}\right)_{\epsilon\searrow 0}\\
&=\mathcal{V}_{k,N}+\epsilon^{\alpha}|\log\epsilon|\left[
 \mathcal{C}_{1}+O\left(\frac{\epsilon^{\frac{N-2k}{k}}}{\epsilon^{\alpha}|\log\epsilon|}\right)_{\epsilon\searrow 0}+o\left(1\right)_{\epsilon\searrow 0}\right]\\
& > \mathcal{V}_{k,N},
\end{aligned}
\end{equation}
for $0<\alpha<(N-2k)/k$ and $\epsilon>0$ small enough.
\subsection{Proof of Step 2}
Let $(v_n)\subset X_1$ be a normalized concentrating sequence at the origin.  It is sufficient to show that, for each $\epsilon>0$,  there are $\delta>0$ and $n_0\in\mathbb{N}$ satisfying
\begin{enumerate}
\item [(i)] $\displaystyle\int_{0}^{\delta}r^{N-1}|v_n|^{k^*+r^{\alpha}}\ud r\le \mathcal{V}_{k,N}+\frac{\epsilon}{2}$,\;\; $\forall \; n\ge n_0$\\

\item [(ii)] $\displaystyle\int_{\delta}^{1}r^{N-1}|v_n|^{k^*+r^{\alpha}}\ud r\le \displaystyle\frac{\epsilon}{2}$, \;\; $\forall \; n\ge n_0$. \\
\end{enumerate}
From Lemma~\ref{r-typeLemma} (see also, \eqref{radial-lemma}), we obtain $C>1$ such that
\begin{equation}\label{Xradial-lemma}
|v_n(r)|\le C r^{\frac{2k-N}{k+1}}, \;\; \forall\; 0<r\le 1.
\end{equation}
In addition, we clearly have
\begin{equation}\nonumber
\lim_{r\rightarrow 0^+}r^{\alpha}\log\left(Cr^{\frac{2k-N}{k+1}}\right)\searrow 0\quad\quad\mbox{and}\quad\quad\lim_{s\rightarrow 0}\frac{\mathrm{e}^{s}-1}{s}=1.
\end{equation}
Hence, we conclude that
\begin{equation}\label{r-contratating-0}
\begin{aligned}
\int_{0}^{\delta}r^{N-1}|v_{n}|^{k^*}\left(|v_{n}|^{r^{\alpha}}-1\right)\ud r &\le \int_{0}^{\delta}r^{N-1}|v_{n}|^{k^*}\left[\exp\left(r^{\alpha}\log\left(Cr^{\frac{2k-N}{k+1}}\right)\right)-1\right]\ud r\\
&\le  C \int_{0}^{\delta}r^{N-1}|v_{n}|^{k^*}r^{\alpha}\left|\log\left(Cr^{\frac{2k-N}{k+1}}\right)\right|\ud r\\
&\le  C_1\delta^{\alpha}\left|\log\left(C\delta^{\frac{2k-N}{k+1}}\right)\right| \int_{0}^{\delta}r^{N-1}|v_{n}|^{k^*}\ud r\\
&\le  C_1\delta^{\alpha}\left|\log\left(C\delta^{\frac{2k-N}{k+1}}\right)\right| \mathcal{V}_{k,N},
\end{aligned}
\end{equation}
by choosing a small enough $\delta>0$. Hence, taking some $\delta=\delta(\alpha, \epsilon,k,N)>0$ small enough such that
\begin{equation}\nonumber
 C_1\delta^{\alpha}\left|\log\left(C\delta^{\frac{2k-N}{k+1}}\right)\right| \mathcal{V}_{k,N}\le \frac{\epsilon}{2}, 
\end{equation} 
we obtain
\begin{equation}\nonumber
\begin{aligned}
\int_{0}^{\delta}r^{N-1}|v_{n}|^{k^*+r^{\alpha}}\ud r & =\int_{0}^{\delta}r^{N-1}|v_{n}|^{k^*}\ud r+\int_{0}^{\delta}r^{N-1}|v_{n}|^{k^*}\left(|v_{n}|^{r^{\alpha}}-1\right)\ud r \\
&\le \mathcal{V}_{k,N}+\frac{\epsilon}{2},
\end{aligned}
\end{equation}
which proves $\mathrm{(i)}$.

As in the proof of  Lemma~\ref{r-typeLemma}, for all $r\in (\delta, 1)$, we obtain
\begin{equation}\nonumber
\begin{aligned}
|v_n(r)|&\le  \int_{r}^{1}|v_n^{\prime}(s)|\ud s = \int_{r}^{1}s^{\frac{N-k}{k+1}}|v_n^{\prime}(s)|s^{-\frac{N-k}{k+1}}\ud s\\
&\le  \left(\int_{\delta}^{1}s^{N-k}|v_n^{\prime}|^{k+1}\ud s\right)^{\frac{1}{k+1}} \left(\int_{r}^{1}s^{-\frac{N-k}{k}}\ud s\right)^{\frac{k}{k+1}}\\
&\le \delta_{n}\frac{1}{r^{\frac{N-2k}{k+1}}},
\end{aligned}
\end{equation}
where
\begin{equation}\nonumber
\delta_{n}=C\left(\int_{\delta}^{1}s^{N-k}|v_n^{\prime}|^{k+1}\ud s\right)^{\frac{1}{k+1}},
\end{equation}
for some $C=C(k,N)$. Since $(v_n)$ is a concentrating sequence at the origin, we have 
$$
\lim _{n\rightarrow\infty}\delta_{n}= 0.
$$
It follows that
\begin{equation}\label{r-concentrating-1}
\begin{aligned}
\int_{\delta}^{1}r^{N-1}|v_{n}|^{k^*+r^{\alpha}}\ud r & \le \int_{\delta}^{1}r^{N-1}\left(\frac{\delta_n}{r^{\frac{N-2k}{k+1}}}\right)^{k^*+r^{\alpha}}\ud r\\
&\le \delta^{k^*}_{n}\int_{\delta}^{1}r^{N-1}\left(\frac{1}{r^{\frac{N-2k}{k+1}}}\right)^{k^*+r^{\alpha}}\ud r\\
&=\delta^{k^*}_{n}C(\delta)\le \frac{\epsilon}{2},
\end{aligned}
\end{equation}
for sufficiently large $n$. This proves $\mathrm{(ii)}$, and consequently, Step~2 holds.
\subsection{Proof of Step 3} Suppose that the supremum $\mathcal{V}_{k,N,\alpha}$ is not attained. Then we are  going to show that every sequence $(v_n)\subset X_1$  satisfying
\begin{equation}\label{max-condiction}
\|v_n\|_{X_1}=1\;\;\;\mbox{and}\;\; \lim_{n\rightarrow\infty}\int_{0}^{1}r^{N-1}|v_n|^{k^{*}+r^{\alpha}}\mathrm{d} r=\mathcal{V}_{k,N,\alpha}
\end{equation}
is necessarily concentrated at the origin in $X_1$. From the compact embedding \eqref{Ebeddings}, up to a subsequence,  we can assume that there exists $v\in X_1$ such that
\begin{equation}\label{three-converg}
v_n\rightharpoonup v\;\;\;\mbox{weakly in}\;\;\; X_1, \; \;\; v_n\rightarrow v\;\;\mbox{in}\;\; L^{q}_{N-1},\; (q<k^*)\;\;\mbox{and}\;\; v_n(r)\rightarrow v(r)\;\;\mbox{a.e}\;\; \mbox{in}\;\; (0,1).
\end{equation}
Let us  denote by $X_1([r_0, 1])$ the space $X_1$ on the interval $[r_0,1]$ instead of $(0,1]$. We claim that the embedding
\begin{equation}\label{emb-rabo}
X_1([r_0, 1]) \hookrightarrow L^q_{N-1}[r_0,1] 
\end{equation}
is compact for any $q\ge k+1$.  To prove \eqref{emb-rabo}, we consider the operator $H: L^{k+1}_{N-k}[r_0,1]\rightarrow L^q_{N-1}[r_0,1]$ defined by $$ H(f)(r)= \int_{r}^{1}f(s)\ud s.$$
Using \cite[Theorem~7.4]{Opic}, for $q\ge k+1$, the operator $H$ is compact if and only if the following  assert holds:
\begin{equation}\label{three-cond}
\left|
\begin{aligned}
&\sup_{r\in (r_0, 1)}F(r)<\infty\\
&\lim_{r\rightarrow r_{0}^{+}}F(r)=0\\
&\lim_{r\rightarrow 1^{-}}F(r)=0,
\end{aligned}
\right.
\end{equation}
where
\begin{equation}\nonumber
F(r)=\left(\int_{r_0}^{r}s^{N-1}\mathrm{d}s\right)^{\frac{1}{q}}\left(\int_{r}^{1}s^{-\frac{N-k}{k}}\mathrm{d}s\right)^{\frac{k}{k+1}}.
\end{equation}
It is easy to see that \eqref{three-cond}  holds. In addition,  the embedding \eqref{emb-rabo} can be seen as the composition $H\circ T$, where  $$T:X_1([r_0, 1])\rightarrow L^{k+1}_{N-k}[r_0,1],\;\;\;\;Tv=-v^{\prime}.$$ Since $T$ is a continuous operator, we conclude the embedding \eqref{emb-rabo} is compact.

Fix $q>k+1$ and $r_0\in (0,1)$. From \eqref{Xradial-lemma}, there is $c_0>0$ depending only on $r_0, q, k,$ and $N$ such that
$$
\sup_{r\in [r_0,1]}|v_n(r)|^{(k^*+r^{\alpha}-1)\frac{q}{q-1}}\le c_0.
$$
Hence, the embedding \eqref{emb-rabo} together with the H\"{o}lder inequality gives
\begin{equation}\label{concentrationS}
\begin{aligned}
\int_{r_0}^{1}r^{N-1}|v_{{n}}|^{k^{*}+r^{\alpha}-1}|v_n-v|\,\mathrm{d}r
&\le c_0^{\frac{q-1}{q}} \left(\int_{r_0}^{1}r^{N-1}|v_{{n}}-v|^{q}\ud r\right)^{\frac{1}{q}}\rightarrow 0.
\end{aligned}
\end{equation}
From the Ekeland's principle \cite[Theorem~3.1]{Ekeland} (cf.\eqref{max-condiction}), we can assume that
\begin{equation}\label{Euler mult}
\lambda_{n}\left(\omega_{N,k}\int_{0}^{1}r^{N-k}|v^{\prime}_n|^{k-1}v_n^{\prime}w^{\prime}\ud r\right) = \int_{0}^{1}r^{N-1}(k^{*}+r^{\alpha})|v_n|^{k^{*}-2+r^{\alpha}}v_n w \ud r+\langle o(1),w \rangle
\end{equation}
for some multiplier $\lambda_n$. Choosing $w=v_n$ one has
\begin{equation}\nonumber
\lambda_{n}
 \ge  k^{*}\int_{0}^{1}r^{N-1}|v_n|^{k^{*}+r^{\alpha}}\ud r+\langle o(1),v_n \rangle,
\end{equation}
and consequently
\begin{equation}\label{lagrange>0}
\liminf_{n} \lambda_n\ge  k^{*}\mathcal{V}_{k,N,\alpha}>0.
\end{equation}
Let $\eta$ be a smooth cut-off function satisfying
\begin{equation}\label{eta-function}
\eta (r)=\left\{\begin{aligned}
&0, \quad\mbox{if}\;\; r\in [0, r_0/2]\\
&1, \quad\mbox{if}\;\; r\in [r_0,1]
\end{aligned}\right..
\end{equation}
Thus, using $w=\eta (v_n-v)$ in \eqref{Euler mult},  \eqref{concentrationS} and \eqref{lagrange>0} yield
\begin{equation}\nonumber
\begin{aligned}
\int_{r_0/2}^{1}r^{N-k}|v^{\prime}_n|^{k-1}v^{\prime}_n(\eta (v_{n}-v))^{\prime}\ud r 
= o(1).
\end{aligned}
\end{equation}
Also, the compact embedding \eqref{Ebeddings} (cf. \eqref{three-converg}) gives
\begin{equation}\nonumber
\int_{r_0}^{1}r^{N-k}|v_n-v|^{k+1} \ud r\le \frac{1}{r^{k+1}_0} \int_{r_0}^{1}r^{N-1}|v_n-v|^{k+1}\ud  r\rightarrow 0.
\end{equation}
Consequently, we get
\begin{equation}\nonumber
\begin{aligned}
o(1)&=\left|\int_{r_0/2}^{1}r^{N-k}|v^{\prime}_n|^{k-1}v^{\prime}_n(\eta (v_{n}-v))^{\prime}\ud r \right|&\\
 &\ge \left|\int_{r_0/2}^{1}r^{N-k}\eta |v^{\prime}_n|^{k-1}v^{\prime}_{n}(v_n-v)^{\prime}\ud r\right|-\left|\int_{r_0/2}^{1}r^{N-k}|v^{\prime}_n|^{k-1}v^{\prime}_n (v_n-v)\eta^{\prime} \ud r\right|\\
&\ge \left|\int_{r_0/2}^{1}r^{N-k}\eta |v^{\prime}_n|^{k-1}v^{\prime}_{n}(v_n-v)^{\prime}\ud r\right|-C\|\eta^{\prime}\|_{\infty}\|v_n\|^{k}_{X_1}\left(\int_{r_0}^{1}r^{N-k}|v_n-v|^{k+1} \ud r\right)^{\frac{1}{k+1}}\\
&= \left|\int_{r_0/2}^{1}r^{N-k}\eta |v^{\prime}_n|^{k-1}v^{\prime}_{n}(v_n-v)^{\prime}\ud r\right|+o(1),
\end{aligned}
\end{equation}
for some constant $C(N,k)>0$. Hence,
\begin{equation}\label{concentrate-quasi}
\int_{r_0/2}^{1}r^{N-k}\eta |v^{\prime}_n|^{k-1}v^{\prime}_{n}(v_n-v)^{\prime}\ud r=o(1).
\end{equation}
In addition, in view of the weak convergence in \eqref{three-converg}, one has
\begin{equation}\label{wear-convergence-mode}
\int_{0}^{1}r^{N-k}\eta|v^{\prime}|^{k-1}v^{\prime}(v_n-v)^{\prime}\ud r=o(1).
\end{equation}
By combining \eqref{concentrate-quasi} and \eqref{wear-convergence-mode} with the elementary inequality
\begin{equation}\nonumber
2^{2-p}|b-a|^{p}\le \left(|b|^{p-2}b-|a|^{p-2}a\right)(b-a), \;\; p\ge 2,\;\; a, b\in \mathbb{R}
\end{equation}
one has
\begin{equation}\label{concentrate-completed}
\int_{r_0}^{1}r^{N-k}|v^{\prime}_n-v^{\prime}|^{k+1}\ud r\le \int_{r_0/2}^{1}r^{N-k}\eta|v^{\prime}_n-v^{\prime}|^{k+1}\ud r\rightarrow 0.
\end{equation}
Since $r_0\in (0,1)$ is arbitrary, up to a subsequence,  we have $v^{\prime}_{n}(r)\rightarrow v^{\prime}(r)$ a.e. in $(0,1)$. Hence,  by Brezis-Lieb type argument (cf.\cite{BL}), we can write
\begin{equation}\label{bl1}
\int_{0}^{1} r^{N-1}|v_n|^{k^{*}+r^{\alpha}}\ud r = \int_{0}^{1} r^{N-1}|v_n-v|^{k^{*}+r^{\alpha}}\ud r+\int_{0}^{1} r^{N-1}|v|^{k^{*}+r^{\alpha}}\ud r +o(1)
\end{equation}
and 
\begin{equation}\label{bl2}
1=\|v_n\|^{k+1}_{X_1}=\|v_n-v\|^{k+1}_{X_1}+\|v\|^{k+1}_{X_1} +o(1).
\end{equation}
Of course, we have $\|v\|_{X_1}\le 1$. We claim that $v=0$ in $X_1$. Arguing by contradiction, we suppose that
\begin{equation}\label{contra-Hyp}
\int_{0}^{1}r^{N-k}|v^{\prime}|^{k+1}\ud r>0.
\end{equation}
If $\|v\|_{X_1}=1$, from \eqref{bl2} we obtain $v_n\rightarrow v$ strongly in $X_1$. In this case, we will prove that $v$ is a maximizer  of $\mathcal{V}_{k,N,\alpha}$, which contradicts our assumption. Indeed, from \eqref{max-condiction} and \eqref{bl1}, it is sufficient to show that
\begin{equation} \label{case1-att}
\limsup_{n}\int_{0}^{1} r^{N-1}|v_n-v|^{k^{*}+r^{\alpha}}\ud r=0.
\end{equation}
By choosing $n$ large enough such that $\|v_n-v\|_{X_1}<1$, \eqref{supremumX} yields
\begin{equation}\nonumber
\frac{1}{\|v_n-v\|^{k^*}_{X_1}}\int_{0}^{1}r^{N-1}|v_n-v|^{k^*+r^{\alpha}}\ud r\le\int_{0}^{1}r^{N-1}\left|\frac{v_n-v}{\|v_n-v\|_{X_1}}\right|^{k^*+r^{\alpha}}\ud r\le  \mathcal{V}_{k,N,\alpha},
\end{equation}
which gives \eqref{case1-att}.

Hence, we can assume $\|v\|_{X_1}<1$. Setting $w_{n}=v_n-v$ and  using \eqref{contra-Hyp} and \eqref{bl2}, we have $\|w_n\|_{X_1}<1$. Hence, \eqref{max-condiction}, \eqref{bl1} and \eqref{bl2} imply
\begin{equation}\nonumber
\begin{aligned}
 \mathcal{V}_{k,N,\alpha}&=\int_{0}^{1} r^{N-1}|w_n|^{k^{*}+r^{\alpha}}\ud r+\int_{0}^{1} r^{N-1}|v|^{k^{*}+r^{\alpha}}\ud r +o(1)\\
&=\int_{0}^{1} r^{N-1}\left|\frac{w_n}{\|w_n\|_{X_1}}\right|^{k^{*}+r^{\alpha}}\|w_n\|_{X_1}^{k^{*}+r^{\alpha}}\ud r\\
&+\int_{0}^{1} r^{N-1}\left|\frac{v}{\|v\|_{X_1}}\right|^{k^{*}+r^{\alpha}}\|v\|_{X_1}^{k^{*}+r^{\alpha}}\ud r +o(1)\\ 
&\le \mathcal{V}_{k,N,\alpha}\left(\|w_n\|_{X_1}^{k^{*}} +\|v\|_{X_1}^{k^{*}}\right)+o(1)\\ 
&= \mathcal{V}_{k,N,\alpha}\left(\left(1-\|v\|^{k+1}_{X_1}+o(1)\right)^{\frac{k^*}{k+1}}+(\|v\|^{k+1}_{X_1})^{\frac{k^*}{k+1}}\right)+o(1)\\
&<\mathcal{V}_{k,N,\alpha},
\end{aligned}
\end{equation}
where we have still used $(1-t)^{k^*/(k+1)}+t^{k^*/(k+1)}<1$, for all $0<t<1$. This contradiction forces $v\equiv 0$ in $X_1$. 

In order to complete the proof of Step~3, is now sufficient to show that $(v_n)$ satisfies the condition
\begin{equation}\label{con-remains}
\int_{r_0}^{1}r^{N-k}|v^{\prime}_{{n}}|^{k+1}\,\mathrm{d}r\rightarrow0,\;\forall\;
r_0>0
\end{equation}
but it is an immediate  consequence of \eqref{concentrate-completed} since  we have proved $v\equiv0$. 
\section{Existence of $k$-admissible Extremals: Proof of Theorem~\ref{thm2}}
\label{sec4}
In order to ensure the existence of an extremal function for the supremum \eqref{max-problem}, we will use the maximizer of  the auxiliary problem \eqref{supremumX}, which is ensured by the Proposition~\ref{prop1}.

Let $v_0\in X_1$ be a maximizer of $\mathcal{V}_{k,N,\alpha}$, which we can assume $v_0\ge 0$  since $|v_0|$ is  still a maximizer. We set
\begin{equation}\nonumber
 u_0(x)=-v_0(|x|), \;\; x\in B.
\end{equation} 
From  \eqref{FX} and \eqref{NormX}, we get
\begin{equation}\label{FX1}
\int_{B}|u_0|^{k^{*}+|x|^{\alpha}}\mathrm{d}x=\omega_{N-1}\int_{0}^{1}r^{N-1}|v_0|^{k^{*}+r^{\alpha}}\mathrm{d}r=\omega_{N-1}\mathcal{V}_{k,N,\alpha}\ge \mathcal{U}_{k,N,\alpha},
\end{equation}
and
\begin{equation}\label{NormX1}
\|u_0\|_{\Phi^{k}_0}=\|v_0\|_{X_1}=1.
\end{equation}
Hence, it is sufficient to show that $u_0$ belongs to $\Phi^{k}_{0,\mathrm{rad}}(B)$.  To show $u_0\in C^{2}(B)$ or equivalently $v_0\in C^{2}[0,1]$, in  the same way as in \cite{MR1051232}, we will adapt the classical De Giorgi-Nash-Moser estimate  to our framework. 

\begin{lemma}\label{Lemma-regular} If $v\in X_1$ is a maximizer of $\mathcal{V}_{k,N,\alpha}$, then
$\sup_{r\in(0,1]}|v(r)|<+\infty$.
\end{lemma}
\begin{proof}
The Lagrange multipliers theorem yields 
\begin{equation}\label{weaksolution}
\int_{0}^{1}r^{N-k}|v^{\prime}|^{k-1}v^{\prime}h^{\prime}\,\mathrm{d}
r=\lambda \int_{0}^{1}r^{N-1}(k^*+r^{\alpha})|v|^{k^*+r^{\alpha}-2}v\, h\,\mathrm{d}r,\;\; \forall\;\; h\in X_1
\end{equation}
where 
\begin{equation}\label{constants}
\lambda=\frac{1}{\omega_{N,k}\int_{0}^{1}r^{N-1}(k^*+r^{\alpha})|v|^{k^*+r^{\alpha}}\,\mathrm{d}r}.
\end{equation}
For $\sigma, L\ge 1$, let $H\in C^{1}[1,\infty)$ such that $H(t)=t^{\sigma}-1$ for $t\in [1,L]$ and $H$ is linear in $[L,\infty)$. Then, we set 
$$
h(r)=\int_{1}^{1+v^{+}}|H^{\prime}(t)|^{k+1}\ud t,\;\; r\in (0,1]
$$
where $v^{+}=\max\left\{v, 0\right\}$. It is clear that $h\in X_1$ and, since $H^{\prime}$ is an increase function, $h\le v^{+}|H^{\prime}(v^{+}+1)|^{k+1}$. It follows from \eqref{weaksolution} that
\begin{equation}\label{Regularity-1}
\begin{aligned}
\int_{0}^{1}r^{N-k}\left|\frac{d}{dr} H(v^{+}+1)\right|^{k+1}\,\mathrm{d}
r & \le C \int_{0}^{1}r^{N-1}v^{+}\left|H^{\prime}(v^{+}+1)\right|^{k+1}|v^{+}|^{k^*+r^{\alpha}-1} \,\mathrm{d}r\\
& \le C \int_{0}^{1}r^{N-1}(v^{+}+1)\left|H^{\prime}(v^{+}+1)\right|^{k+1}(v^{+}+1)^{k^*+r^{\alpha}-1} \,\mathrm{d}r,\\
\end{aligned}
\end{equation}
for some $C>0$. Now, from Lemma~\ref{super-bounded}  we obtain
$$
\int_{0}^{1}r^{N-1}|v|^{k^*+\frac{p}{p-k-1}r^{\alpha}}\ud r<\infty,
$$
for any $p>k+1$. Hence, by choosing $p\in (k+1, k^{*})$ such that $(k^*-k-2)p<k^*(p-k-1)$, we obtain
\begin{equation}\label{Regularity-2}
\begin{aligned}
\int_{0}^{1}r^{N-1}(v^{+}+1)^{(k^*-k-2+r^{\alpha})\frac{p}{p-k-1}} \,\mathrm{d}r &\le 
\int_{0}^{1}r^{N-1}(v^{+}+1)^{k^*+\frac{p}{p-k-1}r^{\alpha}} \,\mathrm{d}r\\
&\le C_1\int_{0}^{1}r^{N-1}|v^+|^{k^*+\frac{p}{p-k-1}r^{\alpha}}\ud r+C_2\\
&\le C
\end{aligned}
\end{equation}
for  some constant $C>0$. Using the H\"{o}lder  inequality together with \eqref{Ebeddings}, \eqref{Regularity-1} and \eqref{Regularity-2}, we obtain for any $p<q<k^*$
\begin{equation}\nonumber
\begin{aligned}
\left(\int_{0}^{1}r^{N-1}|H(v^{+}+1)|^{q}\ud r\right)^{\frac{1}{q}}& \le C\left(\int_{0}^{1}r^{N-k}\left|\frac{d}{dr} H(v^{+}+1)\right|^{k+1}\,\mathrm{d}
r\right)^{\frac{1}{k+1}} \\
&\le C\left(\int_{0}^{1}r^{N-1}[(v^{+}+1)\left|H^{\prime}(v^{+}+1)\right|]^{p}\,\mathrm{d}r\right)^{\frac{1}{p}}\times\\
&\times\left(\int_{0}^{1}r^{N-1}(v^{+}+1)^{(k^*-k-2+r^{\alpha})\frac{p}{p-k-1}} \,\mathrm{d}r \right)^{\frac{p-k-1}{p(k+1)}}\\
& \le C\left(\int_{0}^{1}r^{N-1}[(v^{+}+1)\left|H^{\prime}(v^{+}+1)\right|]^{p}\,\mathrm{d}r\right)^{\frac{1}{p}}.
\end{aligned}
\end{equation}
Letting $L\rightarrow+\infty$, from the definition of $H$, we obtain
\begin{equation}\nonumber
\begin{aligned}
\|v^{+}+1\|_{L^{\chi p\sigma}_{N-1}}=\|v^{+}+1\|_{L^{q\sigma}_{N-1}}& \le (C\sigma)^{\frac{1}{\sigma}}\|v^{+}+1\|_{L^{p\sigma}_{N-1}},
\end{aligned}
\end{equation}
where $\chi =\frac{q}{p}>1$. Setting $\sigma=\chi$, $\sigma=\chi^{2},\cdots, \sigma= \chi^{i}$, an iteration yields 
\begin{equation}\label{Regularity-3}
\begin{aligned}
\|v^{+}+1\|_{L^{\chi^{i} p}_{N-1}}& \le C^{\left(\sum_{j=1}^{i-1}\frac{1}{\chi^{j}}\right)}\chi^{\sum_{j=1}^{i-1}\frac{j}{\chi^{j}}}\|v^{+}+1\|_{L^{\chi p}_{N-1}}\le C\|v^{+}+1\|_{L^{\chi p}_{N-1}},\;\; \forall i\in\mathbb{N}.
\end{aligned}
\end{equation}
Hence, $\sup_{r\in (0,1]} v^{+}(r)<\infty.$  Similarly, by using $v^{-}=-\min\left\{v,0\right\}$  instead of $v^{+}$ in the above argument, one can show that $v^{-}$ is bounded. Consequently, $v$ is bounded in $(0,1]$.
 \end{proof}
Next, we will explicit expressions for  $v^{\prime}_0$ and $v^{\prime\prime}_0$. Following the same argument 
in \cite{Clement-deFigueiredo-Mitidieri}, for each $r\in(0,1)$ and $\rho>0$, we consider the
function $h_{\rho}\in X_1$ given by
\begin{equation}\label{keyfunction}
h_{\rho}(s)=\left\{\begin{aligned}&\;1 &\mbox{if}\quad &0\le s\le r,&\\
&\;1+\frac{1}{\rho}(r-s)&\mbox{if}\quad &r\le s\le r+\rho,&\\
&\;0 &\mbox{if}\quad &s\ge r+\rho.&
\end{aligned}\right.
\end{equation}
Applying \eqref{weaksolution} with $h=h_{\rho}$ and  letting
$\rho\rightarrow 0$, we conclude
\begin{equation}\label{integral-I}
r^{N-k}(-|v_0^{\prime}|^{k-1}v_0^{\prime})=\lambda\int_{0}^{r}s^{N-1}(k^*+r^{\alpha})|v_0|^{k^*+s^{\alpha}-2}v_0\,\mathrm{d}s,\quad\mbox{a.e
on}\quad [0,1].
\end{equation}
Since we are assuming $v_0\ge 0$, we can write 
\begin{equation}\label{integralsolution}
-v^{\prime}_0(r)= \left[I(r)\right]^{\frac{1}{k}};\;\;\mbox{with}\;\; I(r)=\frac{\lambda}{r^{N-k}}\int_{0}^{r}s^{N-1}|v_0|^{k^*+s^{\alpha}-1}\,\mathrm{d}s.
\end{equation}
Consequently, we have $v_0\in C^{2}(0,1]$. In addition, from Lemma~\ref{Lemma-regular}, we get
\begin{equation}\nonumber
\lim_{r\rightarrow 0^{+}}I(r)=0.
\end{equation}
Hence, from \eqref{integralsolution} 
$$\lim_{r\rightarrow 0^{+}}v_0^{\prime}(r)=0$$
and thus $v_0\in C^{1}[0,1]$.

In order to get $v_0\in C^{2}[0,1]$, we firstly observe that
\begin{equation}\label{chiI}
\begin{aligned}
\frac{I^{\prime}(r)}{I(r)}=-\frac{N-k}{r}+\frac{r^{N-1}|v_0|^{k^*+r^{\alpha}-1}}{\int_{0}^{r}s^{N-1}|v_0|^{k^*+s^{\alpha}-1}\ud s}, \;\;\forall \, r\in (0,1].
\end{aligned}
\end{equation} 
From \eqref{integralsolution} and \eqref{chiI} it follows that
\begin{equation} \label{bi-grad}
\begin{aligned}
-v^{\prime\prime}_0(r) &= \frac{[I(r)]^{\frac{1}{k}}}{k}\frac{I^{\prime}(r)}{I(r)}\\
&=-\frac{v^{\prime}_0(r)}{k}\left[-\frac{N-k}{r}+\frac{r^{N-1}|v_0|^{k^*+r^{\alpha}-1}}{\int_{0}^{r}s^{N-1}|v_0|^{k^*+s^{\alpha}-1}\ud s}\right],
\end{aligned}
\end{equation} 
for all $r\in (0,1]$. From \eqref{integralsolution}, $v_0$ is a decreasing function with $v_0(0)>0$ and we also obtain 
\begin{equation}\label{frac-frac}
\lim_{r\rightarrow 0^{+}}-\frac{v^{\prime}_0(r)}{r}=\lim_{r\rightarrow 0^{+}}\left({\frac{I(r)}{r^{k}}}\right)^{\frac{1}{k}}=\left({\frac{\lambda}{N}|v_0(0)|^{k^*-1}}\right)^{\frac{1}{k}}>0.
\end{equation}
In addition, the identity \eqref{integralsolution} yields
\begin{equation}\label{frac-cabu}
\begin{aligned}
\frac{r^{N-1}|v_0|^{k^*+r^{\alpha}-1}v^{\prime}_{0}}{\int_{0}^{r}s^{N-1}|v_0|^{k^*+s^{\alpha}-1}\ud s}& =\frac{\lambda r^{N-1}|v_0|^{k^*+r^{\alpha}-1}v^{\prime}_{0}}{r^{N-k}(-v^{\prime}_0)^{k}}\\
&= -\lambda\left(-\frac{r}{v^{\prime}_0}\right)^{k-1}|v_0|^{k^*+r^{\alpha}-1},
\end{aligned}
\end{equation}
for small enough $r>0$. From \eqref{bi-grad}, \eqref{frac-frac} and \eqref{frac-cabu}, one gets that there exists  $\displaystyle\lim_{r\rightarrow 0^{+}}v_0^{\prime\prime}(r)$, and thus $v_0\in C^{2}[0,1]$ holds.

Now, in order to guarantee $u_0\in \Phi^{k}_{0,\mathrm{rad}}(B)$, it is enough to show that
\begin{equation}\nonumber
F_j[u_0]\ge 0\;\;\mbox{in}\;\; B,\;\;\;  \forall\; 1\le j\le k. 
\end{equation}
But, using the $k$-Hessian radial expression (cf. \eqref{Kourpradial}) and the definition $u_0(x)=-v_0(|x|)$, we can reduce the  above assert to the following
\begin{equation}\label{v-kadmi}
\left(r^{N-j}(-v_0^{\prime})^{j}\right)^{\prime}\ge 0,\;\; \forall\; 1\le j\le k\;\;\mbox{and}\;\; r\in (0,1].
\end{equation}
By using the expressions in \eqref{integralsolution} and \eqref{chiI}, it is easy to see that 
\begin{equation}\nonumber
\begin{aligned}
&\left(r^{N-j}(-v_0^{\prime})^{j}\right)^{\prime}=
\left(r^{N-j}\left[I(r)\right]^{\frac{j}{k}}\right)^{\prime}\\
&=(N-j)r^{N-j-1}\left[I(r)\right]^{\frac{j}{k}}+r^{N-j}\frac{j}{k}\left[I(r)\right]^{\frac{j}{k}}\frac{I^{\prime}(r)}{I(r)}\\
&=r^{N-j}\left[I(r)\right]^{\frac{j}{k}}\left[\frac{N-j}{r}+\frac{j}{k}\frac{I^{\prime}(r)}{I(r)}\right].
\end{aligned}
\end{equation}
To prove \eqref{v-kadmi}, it is enough to show that
$$
\left[\frac{N-j}{r}+\frac{j}{k}\frac{I^{\prime}(r)}{I(r)}\right]\ge 0,\;\; \forall\; 1\le j\le k\;\;\mbox{and}\;\; r\in (0,1].
$$
However, from \eqref{chiI}, we can write
\begin{equation}\nonumber
\begin{aligned}
&\frac{N-j}{r}+\frac{j}{k}\frac{I^{\prime}(r)}{I(r)}=\frac{N-j}{r}-\frac{j}{k}\frac{N-k}{r}+\frac{j}{k}\frac{r^{N-1}|v_0|^{k^*+r^{\alpha}-1}}{\int_{0}^{r}s^{N-1}|v_0|^{k^*+s^{\alpha}-1}\ud s}\\
&=\frac{N-k}{r}\left[\frac{N-j}{N-k}-\frac{j}{k}\right]+\frac{j}{k}\frac{r^{N-1}|v_0|^{k^*+r^{\alpha}-1}}{\int_{0}^{r}s^{N-1}|v_0|^{k^*+s^{\alpha}-1}\ud s}
\end{aligned}
\end{equation}
which is non-negative, for all $1\le j\le k$ and $r\in (0,1]$. 

\section{$k$-admissible solution to the related supercritical equation}
\label{sec5}
To find a radially symmetric solutions for the $k$-Hessian equation \eqref{ourp}, or equivalently a solution to \eqref{Kourpradial}, as in \cite{JDE2019,Dai}, we introduce the following quasilinear equation
\begin{equation}\label{ourpradial}
\left\{\begin{aligned}
&\left.\begin{aligned}
&-\mathrm{C}^{k}_{N}\left(r^{N-k}|v^{\prime}|^{k-1}v^{\prime}\right)^{\prime}=Nr^{N-1} v^{k^{*}+r^{\alpha}-1}\\
&v>0 \\
\end{aligned}\right\}&\mbox{in}&\;\;(0,1)\\
&\;v(1)=0, \; v^{\prime}(0)=0.
\end{aligned}\right.
\end{equation}
In order to show the existence of a solution for \eqref{ourpradial}, we will follow closely the argument in \cite{MR3514752}. Indeed, we will apply the variant of the well-known mountain pass theorem of Ambrosetti
and Rabinowitz without the Palais-Smale condition in \cite[Theorem 2.2]{MR0709644} to get a nontrivial critical point for the functional
\begin{equation}\label{I-func}
I(v)=\frac{1}{k+1}\int_{0}^{1}r^{N-k}|v^{\prime}|^{k+1}\ud r-\tau\int_{0}^{1}\frac{r^{N-1}}{k^*+r^{\alpha}}(v^{+})^{k^{*}+r^{\alpha}}\ud r:\;\;\; X_1\rightarrow\mathbb{R},
\end{equation}
where $v^{+}=\max\left\{v,0\right\}$  and $\tau =N/\mathrm{C}^{k}_{N}$.

It is sufficient to prove the following four steps:

\paragraph{(1)} The level  $\left(\frac{1}{k+1}-\frac{1}{k^*}\right)\left(\frac{\mathrm{C}^{k}_{N}}{N}\right)^{\frac{k+1}{k^*-k-1}}S^{\frac{N}{2k}}$
is a noncompactness level for the functional $I$.\\

\paragraph{(2)} The mountain-pass level $c_{MP}$ of the functional $I$ satisfies
$$
0<c_{MP}<\left(\frac{1}{k+1}-\frac{1}{k^*}\right)\left(\frac{\mathrm{C}^{k}_{N}}{N}\right)^{\frac{k+1}{k^*-k-1}}S^{\frac{N}{2k}}.
$$
\paragraph{(3)} The equation \eqref{ourpradial} possesses a weak solution $v\in X_1$.\\

\paragraph{(4)} If $v\in X_1$ is a weak solution for \eqref{ourpradial}, then $v\in C^{2}[0,1]$ and $u(x)=-v(|x|),\,x\in B$ is a radially symmetric $k$-admissible solution of the equation \eqref{ourp}.\\

In the next Lemmas, we will always consider the family $(w_{\epsilon})_{\epsilon>0}$ in \eqref{we} with the suitable choice
\begin{equation}\label{B-valor}
\mathcal{C}=\left(\frac{\mathrm{C}^{k}_{N}}{N}\right)^{\frac{1}{k^{*}-k-1}}=\left(\frac{1}{\tau}\right)^{\frac{1}{k^*-k-1}}.
\end{equation}
\begin{lemma} The level  $\left(\frac{1}{k+1}-\frac{1}{k^*}\right)\left(\frac{\mathrm{C}^{k}_{N}}{N}\right)^{\frac{k+1}{k^*-k-1}}S^{\frac{N}{2k}}$
is a noncompactness level for the functional $I$.
\end{lemma}
\begin{proof}
Using Lemma~\ref{lemmaJDE} and the estimate \eqref{lemma-estimate2-sharp}, we can write
\begin{equation}\nonumber
\begin{aligned}
I(w_{\epsilon})
&=\left(\frac{1}{k+1}-\frac{\tau\mathcal{C}^{k^*-k-1}}{k^*}\right)\mathcal{C}^{k+1}S^{\frac{N}{2k}}+O(\epsilon^{\frac{N-2k}{k}})_{\epsilon\searrow 0}+O(\epsilon^{\alpha}\log\epsilon)_{\epsilon\searrow 0}\\
&\rightarrow \left(\frac{1}{k+1}-\frac{1}{k^*}\right)\left(\frac{\mathrm{C}^{k}_{N}}{N}\right)^{\frac{k+1}{k^*-k-1}}S^{\frac{N}{2k}}
\end{aligned}
\end{equation}
where we have used our choice in \eqref{B-valor}. Further, for any $\delta>0$ we can write
\begin{equation}\nonumber
\begin{aligned}
\int_{\delta}^{\infty}r^{N-k}|(v^{*}_{\epsilon})^{\prime}|^{k+1}\ud r=\left(\frac{\hat{c}(N-2k)}{k}\right)^{k+1}\int_{\delta/\epsilon}^{\infty}\frac{s^{1-\frac{N}{k}}}{(1+s^{-2})^{\frac{N}{2k}(k+1)}}=O(\epsilon^{\frac{N-2k}{k}})_{\epsilon\searrow 0}.
\end{aligned}
\end{equation}
For $\delta<r_0$, we have $\eta
\equiv1$ in $(0,\delta)$. Hence, from \eqref{SSpu}
\begin{equation}\nonumber
\begin{aligned}
\int_{0}^{\delta}r^{N-k}|w^{\prime}_{\epsilon}|^{k+1}\ud r&=\mathcal{C}^{k+1}\int_{0}^{\delta}r^{N-k}|(v^{*}_{\epsilon})^{\prime}|^{k+1}\ud r\\
&= \mathcal{C}^{k+1}\left(S^{\frac{N}{2k}}+O(\epsilon^{\frac{N-2k}{k}})_{\epsilon\searrow 0}\right)\\
&\rightarrow \left(\frac{\mathrm{C}^{k}_{N}}{N}\right)^{\frac{k+1}{k^*-k-1}}S^{\frac{N}{2k}},
\end{aligned}
\end{equation}
as $\epsilon\rightarrow 0$.  This together with Lemma~\ref{lemmaJDE} yields
\begin{equation}\nonumber
\begin{aligned}
\lim_{\epsilon\rightarrow0}\int_{\delta}^{1}r^{N-k}|w^{\prime}_{\epsilon}|^{k+1}\ud r=0,
\end{aligned}
\end{equation}
for any $0<\delta<1$. In addition, 
\begin{equation}\nonumber
\begin{aligned}
\int_{0}^{1}r^{N-1}|w_{\epsilon}|^{k+1}\ud r &\le \mathcal{C}^{k+1}\int_{0}^{1}r^{N-1}|v^{*}_{\epsilon}|^{k+1}\ud r\\
& =(\mathcal{C}\hat{c})^{k+1}\epsilon^{-N+2k}\int_{0}^{1}r^{N-1}\left(1+\left(\frac{r}{\epsilon}\right)^2\right)^{-\frac{N-2k}{2k}(k+1)}\ud r\\
& =(\mathcal{C}\hat{c})^{k+1}\epsilon^{2k}\int_{0}^{1/\epsilon}\frac{s^{N-1}}{\left(1+s^2\right)^{\frac{N-2k}{2k}(k+1)}}\ud s\rightarrow 0,\\
\end{aligned}
\end{equation}
as $\epsilon\rightarrow 0$. From the compact embedding \eqref{Ebeddings}, up to a subsequence,  we conclude that  $w_{\epsilon}\rightharpoonup 0$ weakly in $X_1$ and thus $(w_{\epsilon})$ is concentrating at the origin and
does not contain a strongly convergent subsequence.
\end{proof}
\noindent According with \cite{JDE2019}, the functional $I$ belongs to $C^{1}(X_1,\mathbb{R})$ and satisfies the  mountain-pass structure.
This allow us to define the mountain-pass level 
\begin{equation}\label{c-minimax-level}
0<c_{MP}:=\inf_{\gamma\in\Gamma}\max_{u\in\gamma}I(u),
\end{equation}
where
$$
\Gamma=\left\{\gamma\in C([0,T],X_{1})\;:\; \gamma(0)=0,\; \gamma(T)=Tw_{\epsilon} \right\}
$$
with $T>0$ large enough so that $I(Tw_{\epsilon})\le 0$. 
\begin{lemma}\label{mpl-sharp} Suppose $0<\alpha<(N-2k)/k$. Then, the mountain-pass level $c_{MP}$ satisfies
\begin{equation}\nonumber
0<c_{MP}<\left(\frac{1}{k+1}-\frac{1}{k^*}\right)\left(\frac{\mathrm{C}^{k}_{N}}{N}\right)^{\frac{k+1}{k^*-k-1}}S^{\frac{N}{2k}}.
\end{equation}
\end{lemma}
\begin{proof}
For each $\epsilon>0$, there exists $t_{\epsilon}>0$ such that
\begin{equation}\nonumber
c_{MP}\le \max_{t\in [0,T]}I(tw_{\epsilon})=I(t_{\epsilon}w_{\epsilon}).
\end{equation}
Next, we will analyze the behavior of $(t_{\epsilon})_{\epsilon>0}$. Since $\frac{d}{dt}I(tw_{\epsilon})|_{t=t_{\epsilon}}=0$, we obtain
\begin{equation}\label{singular-grad}
\begin{aligned}
\int_{0}^{1}r^{N-k}|w^{\prime}_{\epsilon}|^{k+1}\ud r=\tau t^{k^*-k-1}_{\epsilon}\int_{0}^{1}t^{r^{\alpha}}_{\epsilon}r^{N-1}|w_{\epsilon}|^{k^*+r^{\alpha}}\ud r.
\end{aligned}
\end{equation}
Using Lemma~\ref{lemmaJDE}-$(a)$ together with Lemma~\ref{lemma-sharper}, we obtain
\begin{equation}\label{old-est}
\int_{0}^{1}r^{N-k}|w^{\prime}_{\epsilon}|^{k+1}\ud r=\mathcal{C}^{k+1}S^{\frac{N}{2k}}+O(\epsilon^{\frac{N-2k}{k}})_{\epsilon\searrow0}
\end{equation}
and (cf.\eqref{lemma-estimate})
\begin{equation}\label{sharp-est}
\int_{0}^{1}r^{N-1}|w_{\epsilon}|^{k^*+r^{\alpha}}\ud r = \mathcal{C}^{k^*}S^{\frac{N}{2k}}+ \mathcal{C}_{1}|\log\epsilon|\epsilon^{\alpha}+o\left(\epsilon^{\alpha}|\log\epsilon|\right)_{\epsilon\searrow0}.
\end{equation}
We claim that
\begin{equation}\label{te-limite}
t_{\epsilon}\rightarrow 1,\;\;\mbox{as}\;\; \epsilon\rightarrow 0.
\end{equation}
Indeed, if $
\limsup_{\epsilon\rightarrow0}t_{\epsilon}>1
$
it is possible to choose $\kappa>1$ and a subsequence $(t_{\epsilon_{i}})$ such that $\epsilon_{i}\rightarrow 0$ as $i\rightarrow\infty$ and $t_{\epsilon_{i}}>\kappa$ for all $i$. Using \eqref{B-valor}, \eqref{singular-grad}, \eqref{old-est} and \eqref{sharp-est}, we deduce
\begin{equation}\nonumber
\begin{aligned}
S^{\frac{N}{2k}}+O(\epsilon^{\frac{N-2k}{k}}_{i})_{\epsilon_i\searrow0}&\ge\tau\left(\kappa\mathcal{C}\right)^{k^*-k-1}\left(S^{\frac{N}{2k}}+O\left(\epsilon^{\alpha}_{i}|\log\epsilon_{i}|\right)_{\epsilon_{i}\searrow0}\right)\\
&=\kappa^{k^*-k-1}\left(S^{\frac{N}{2k}}+O\left(\epsilon^{\alpha}_{i}|\log\epsilon_{i}|\right)_{\epsilon_i\searrow0}\right) ,\;\;\forall\, i.
\end{aligned}
\end{equation}
Letting $i\rightarrow\infty$, we obtain $\kappa\le 1$
which is a contradiction. Hence, 
$$
\limsup_{\epsilon\rightarrow0}t_{\epsilon}\le 1.
$$
Analogously, if 
$
\liminf_{\epsilon\rightarrow0}t_{\epsilon}<1
$
there are $\kappa<1$ and a subsequence $(t_{\epsilon_{i}})$ such that $\epsilon_{i}\rightarrow 0$ as $i\rightarrow\infty$ and $t_{\epsilon_{i}}<\kappa$ for all $i$. It also follows from \eqref{B-valor}, \eqref{singular-grad}, \eqref{old-est} and \eqref{sharp-est} that
\begin{equation}\nonumber
S^{\frac{N}{2k}}+O(\epsilon^{\frac{N-2k}{k}}_{i})_{\epsilon_i\searrow0}\le\kappa^{k^*-k-1}\left(S^{\frac{N}{2k}}+O\left(\epsilon^{\alpha}_{i}|\log\epsilon_{i}|\right)_{\epsilon_i\searrow0}\right),\;\;\forall\, i
\end{equation}
which is also a contradiction and the proof of \eqref{te-limite} is completed.

It is clear that  $k^{*}-k-1\le k^{*}-k-1+r^{\alpha}\le k^{*}- k$ and $1/2<t_{\epsilon}<3/2$, for any $r\in [0,1]$ and $\epsilon>0$ small enough. Inspired by \cite{Ngu}, we consider the following auxiliary function defined on $[k^*-k-1,k^*-k]\times[1/2,3/2]$
\begin{equation}
f(q,t)=\left\{\begin{aligned}
&\frac{t^q-1}{t-1}\;\;&\mbox{if}&\;\; t\not =1,\\
& q\;\;&\mbox{if}&\;\;t=1.
\end{aligned}\right.
\end{equation}
The function $f$ is continuous and $f>0$ on $[k^*-k-1,k^*-k]\times[1/2,3/2]$. Hence, 
$$
C_0=\inf\left\{f(q,t)\,:\, (q,t)\in [k^*-k-1,k^*-k]\times[1/2,3/2] \right\}>0.
$$
From  from \eqref{B-valor}, \eqref{singular-grad}, \eqref{old-est} and \eqref{sharp-est}, we can write
\begin{equation}
\begin{aligned}
O(\epsilon^{\alpha}\log\epsilon)_{\epsilon\searrow0}+O\left(\epsilon^{\frac{N-2k}{k}}\right)_{\epsilon\searrow0}&
=\left|\int_{0}^{1}\left(t^{k^*-k-1+r^{\alpha}}_{\epsilon}-1\right)r^{N-1}|w_{\epsilon}|^{k^*+r^{\alpha}}\ud r\right|\\
&=|t_{\epsilon}-1|\int_{0}^{1}f(k^*-k-1+r^{\alpha},t_{\epsilon})r^{N-1}|w_{\epsilon}|^{k^*+r^{\alpha}}\ud r\\
&\ge C_{0}|t_{\epsilon}-1|\left( \mathcal{C}^{k^*}S^{\frac{N}{2k}}+ \mathcal{C}_{1}|\log\epsilon|\epsilon^{\alpha}+o\left(\epsilon^{\alpha}|\log\epsilon|\right)_{\epsilon\searrow 0}\right).
\end{aligned}
\end{equation}
Hence, since we are supposing $\alpha<(N-2k)/k$, we get $t_{\epsilon}=1+R_{\epsilon}$, with $R_{\epsilon}=O(\epsilon^{\alpha}\log\epsilon)_{\epsilon\searrow 0}$. In particular,  using Taylor's expansion
\begin{equation}\label{1Rconst}
\begin{aligned}
(1+R_{\epsilon})^{k+1}&= 1+(k+1)R_{\epsilon}+O(R^{2}_{\epsilon})_{\epsilon\searrow 0}
\end{aligned}
\end{equation}
and
\begin{equation}\nonumber
\begin{aligned}
(1+R_{\epsilon})^{k^*+r^{\alpha}}
&= 1+(k^*+r^{\alpha})R_{\epsilon}+O(R^{2}_{\epsilon})_{\epsilon\searrow 0}.
\end{aligned}
\end{equation}
Therefore, this last estimate together \eqref{sharp-est} ensures 
\begin{equation}\label{peteleco1}
\begin{aligned}
\int_{0}^{1}\frac{(1+R_{\epsilon})^{k^*+r^{\alpha}}-1}{k^*+r^{\alpha}}r^{N-1}|w_{\epsilon}|^{k^*+r^{\alpha}}\ud r& =R_{\epsilon}\int_{0}^{1}r^{N-1}|w_{\epsilon}|^{k^*+r^{\alpha}}\ud r\\
&=R_{\epsilon}\mathcal{C}^{k^*}S^{\frac{N}{2k}}+O(R^{2}_{\epsilon})_{\epsilon\searrow 0}.
\end{aligned}
\end{equation}
Combining \eqref{old-est}, \eqref{1Rconst} and \eqref{peteleco1} with \eqref{lemma-estimate2-sharp}, we get
\begin{equation}
\begin{aligned}
\tau^{-1}c_{MP}&\le \tau^{-1}I(t_{\epsilon}w_{\epsilon})=\frac{t^{k+1}_{\epsilon}}{k+1}\left(\frac{1}{\tau}\int_{0}^{1}r^{N-k}|w_{\epsilon}^{\prime}|^{k+1}\ud r\right)-\int_{0}^{1}\frac{r^{N-1}}{k^*+r^{\alpha}}|t_{\epsilon}w_{\epsilon}|^{k^{*}+r^{\alpha}}\ud r\\
&=\frac{(1+R_{\epsilon})^{k+1}}{k+1}\left(\mathcal{C}^{k^*}S^{\frac{N}{2k}}+O\left(\epsilon^{\frac{N-2k}{k}}\right)_{\epsilon\searrow 0}\right)- 
\int_{0}^{1}\frac{(1+R_{\epsilon})^{k^*+r^{\alpha}}-1}{k^*+r^{\alpha}}r^{N-1}|w_{\epsilon}|^{k^*+r^{\alpha}}\ud r\\
& \quad-\int_{0}^{1}\frac{r^{N-1}}{k^*+r^{\alpha}}|w_{\epsilon}|^{k^{*}+r^{\alpha}}\ud r\\
&=\frac{(1+R_{\epsilon})^{k+1}}{k+1}\left(\mathcal{C}^{k^*}S^{\frac{N}{2k}}+O\left(\epsilon^{\frac{N-2k}{k}}\right)_{\epsilon\searrow 0}\right)-\left(R_{\epsilon}\mathcal{C}^{k^*}S^{\frac{N}{2k}}+O(R^{2}_{\epsilon})_{\epsilon\searrow 0}\right)\\
&\quad\quad - \frac{\mathcal{C}^{k^*}S^{\frac{N}{2k}}}{k^*}-\frac{\mathcal{C}_{1}}{k^*}|\log\epsilon|\epsilon^{\alpha}+o\left(\epsilon^{\alpha}|\log\epsilon|\right)_{\epsilon\searrow0}\\
&= \left(\frac{1}{k+1}-\frac{1}{k^*}\right)\mathcal{C}^{k^*}S^{\frac{N}{2k}}-\frac{\mathcal{C}_1}{k^*}\epsilon^{\alpha}|\log\epsilon| + O\left(\epsilon^{\frac{N-2k}{k}}\right)_{\epsilon\searrow 0}+O(R^{2}_{\epsilon})_{\epsilon\searrow 0}+o(\epsilon^{\alpha}\log \epsilon)_{\epsilon\searrow 0}.
\end{aligned}
\end{equation}
Hence, putting \eqref{B-valor} 
\begin{equation}
\begin{aligned}
c_{MP}&\le \left(\frac{1}{k+1}-\frac{1}{k^*}\right)\left(\frac{\mathrm{C}^{k}_{N}}{N}\right)^{\frac{k+1}{k^*-k-1}}S^{\frac{N}{2k}}\\
&\quad +\epsilon^{\alpha}|\log\epsilon|\left[-\frac{\mathcal{C}_1}{k^*}+O\left(\frac{\epsilon^{\frac{N-2k}{k}}}{\epsilon^{\alpha}|\log\epsilon|}\right)_{\epsilon\searrow 0}+O\left(\frac{R^{2}_{\epsilon}}{\epsilon^{\alpha}|\log\epsilon|}\right)_{\epsilon\searrow 0}+o(1)_{\epsilon\searrow 0}\right],
\end{aligned}
\end{equation}
which completes the proof since $0<\alpha<(N-2k)/k$.
\end{proof}
Taking into account Lemma~\ref{mpl-sharp} and since $I$ has the mountain-pass structure, we can apply \cite[Theorem 2.2]{MR0709644} to get a Palais-Smale sequence $(v_n)\subset X_1$ for the functional $I$ at level $c_{MP}$. That is,
\begin{equation}\label{PS-seq1}
I(v_n)\rightarrow c_{MP}<\left(\frac{1}{k+1}-\frac{1}{k^*}\right)\left(\frac{\mathrm{C}^{k}_{N}}{N}\right)^{\frac{k+1}{k^*-k-1}}S^{\frac{N}{2k}},\;\; \mbox{as}\;\;n\rightarrow \infty
\end{equation}
and
\begin{equation}\label{PS-seq2}
I^{\prime}(v_n)\varphi=\int_{0}^{1}r^{N-k}|v^{\prime}_n|^{k-1}v^{\prime}_{n}\varphi^{\prime}\ud r -\tau \int_{0}^{1}r^{N-1}(v^{+}_n)^{k^*+r^{\alpha}-1}\varphi\ud r\rightarrow 0, \;\; \mbox{as}\;\;n\rightarrow \infty
\end{equation}
for any $\varphi\in X_1$.
\begin{lemma}
 The equation \eqref{ourpradial} possesses a weak solution $v\in X_1$.
\end{lemma}
\begin{proof}
Denote
$$
F(r,s)=\frac{s^{k^{*}+r^{\alpha}}}{k^*+r^{\alpha}}, \;\; s\ge 0\;\;\mbox{and}\;\; r\in [0,1].
$$
It is easy to see that $F$ satisfies the well-known Ambrosetti-Rabinowitz condition: there exists $\xi\in (k+1,k^*)$ such that
\begin{equation}\label{AR}
\xi F(r,s)\le sf(r,s),\;\; \forall\; s\ge 0\;\;\mbox{and}\;\; r\in [0,1]
\end{equation} 
where $f(r,s)=\frac{\partial F}{\partial s}(r,s)$. From \eqref{AR} we have
\begin{equation}\nonumber
\begin{aligned}
|\xi I(v_n)-I^{\prime}(v_n)v_n|
&\ge \left(\frac{\xi}{k+1}-1\right)\int_{0}^{1}r^{N-k}|v^{\prime}_n|^{k+1}\ud r.
\end{aligned}
\end{equation}
Hence, it follows from \eqref{PS-seq1} and \eqref{PS-seq2} that
\begin{equation}\nonumber
\begin{aligned}
C_1+\epsilon_{n}\|v_n\|_{X_1}
&\ge C_2\|v_n\|^{k+1}_{X_1},
\end{aligned}
\end{equation}
where $C_1, C_{2}>0$ and $\epsilon_{n}\rightarrow 0$ as $n\rightarrow\infty$. It follows that 
$(v_n)$ is bounded sequence in $X_1$.  In addition, from \eqref{Ebeddings}, up to a subsequence, we can assume that there exists $v\in X_1$ such that
 \begin{equation}\label{weak-3}
 v_n\rightharpoonup v\;\;\mbox{in}\;\; X_1,\;\;  v_n\rightarrow v\;\;\mbox{in}\;\; L^{q}_{N-1},\;\forall\, q\in [1,k^*)\;\;\mbox{and}\;\; v_n(r)\rightarrow v(r)\;\;\mbox{a.e in}\;\;(0,1).
 \end{equation}
It is standard to show (see for instance  \cite{CCM02017,Djairo-Miy}) that $v\in X_1$ solves the weak equation 
\begin{equation}\label{weak-equation}
\int_{0}^{1}r^{N-k}|v^{\prime}|^{k-1}v^{\prime}\varphi^{\prime}\ud r=\tau \int_{0}^{1}r^{N-1} (v^+)^{k^{*}+r^{\alpha}-1}\varphi\,\ud r, \;\;\mbox{for all}\;\;\varphi\in X_1.
\end{equation}
Using the test function \eqref{keyfunction}, we can write (cf.\eqref{integral-I})
\begin{equation}\label{integral equation- II}
-(r^{N-k}|v^{\prime}|^{k-1}v^{\prime})=\tau\int_{0}^{r}s^{N-1} (v^+)^{k^{*}+s^{\alpha}-1}\,\ud s, \quad\mbox{a.e
on}\quad (0,1).
\end{equation}
Hence,  $v$ is a non-increasing function with $v(1)=0$. Then, either $v>0$ or $v\equiv 0$ in $(0,1)$. It remains to prove that $v\not\equiv 0$. By contradiction, suppose $v\equiv 0$. Thus, taking into account \eqref{weak-3} and \eqref{emb-rabo}, we can apply the same argument in the proof of Step~3, Section~\ref{sec3} (cf.\eqref{concentrate-completed}) to show that
\begin{equation}\label{cI-cpt}
\int_{r_0}^{1}r^{N-k}|v^{\prime}_n|^{k+1}\ud r\rightarrow 0, \;\;\forall r_0>0.
\end{equation}
Define
\begin{equation}
I_0(w)=\frac{1}{k+1}\int_{0}^{1}r^{N-k}|w^{\prime}|^{k+1}\ud r-\frac{\tau}{k^*}\int_{0}^{1}r^{N-1}(w^{+})^{k^{*}}\ud r:\;\;\; X_1\rightarrow\mathbb{R}.
\end{equation}
We claim that
\begin{equation}\label{II0}
I(v_n)=I_0(v_n)+o(1).
\end{equation}
First, for any $r_0\in (0,1)$, we write
\begin{equation}\label{I-I0}
\begin{aligned}
|I(v_n)-I_0(v_n)|&=\left|\frac{\tau}{k^*}\int_{0}^{1}r^{N-1}(v^{+}_{n})^{k^{*}}\ud r-\tau\int_{0}^{1}\frac{r^{N-1}}{k^*+r^{\alpha}}(v^{+}_{n})^{k^{*}+r^{\alpha}}\ud r\right|\\
& \le C_1\left|\int_{0}^{r_0}r^{N-1}(v^{+}_{n})^{k^{*}}((v^{+}_{n})^{r^{\alpha}}-1)\ud r\right|+C_2\left|\int_{r_0}^{1}r^{N-1}(v^{+}_{n})^{k^{*}}((v^{+}_{n})^{r^{\alpha}}-1)\ud r\right|.
\end{aligned}
\end{equation}
Now, from \eqref{cI-cpt}, arguing as in \eqref{r-concentrating-1}, we can see that
\begin{equation}\nonumber
\begin{aligned}
&\int_{r_0}^{1}r^{N-1}(v^{+}_n)^{k^*+r^{\alpha}}\ud r=o(1)\\
&\int_{r_0}^{1}r^{N-1}(v^{+}_n)^{k^*}\ud r=o(1).
\end{aligned}
\end{equation}
It  follows that
\begin{equation}\label{I-I0partI}
\int_{r_0}^{1}r^{N-1}(v^{+}_{n})^{k^{*}}((v^{+}_{n})^{r^{\alpha}}-1)\ud r=o(1).
\end{equation}
Further, as in \eqref{r-contratating-0}, we can write 
\begin{equation}\nonumber
\begin{aligned}
&\int_{\left\{v_n\ge1\right\}\cap(0,r_0)}r^{N-1}(v^{+}_n)^{k^*}((v^{+}_n)^{r^{\alpha}}-1)\ud r=o(1)
\end{aligned}
\end{equation}
for all $r_0>0$ small enough. Consequently
\begin{equation}\label{I-I0partII}
\begin{aligned}
\left|\int_{0}^{r_0}r^{N-1}(v^{+}_{n})^{k^{*}}((v^{+}_{n})^{r^{\alpha}}-1)\ud r\right|& \le \int_{\left\{v_n<1\right\}\cap(0,r_0)}r^{N-1}(v^{+}_n)^{k^*}(1-(v^{+}_n)^{r^{\alpha}})\ud r+o(1)\\
&\le\int_{\left\{v_n<1\right\}\cap(0,r_0)}r^{N-1}\ud r+o(1)\\
& \le \frac{r^{N}_0}{N}+o(1)
\end{aligned}
\end{equation}
for $r_0>0$ small enough. Hence, \eqref{II0} follows from \eqref{I-I0}, \eqref{I-I0partI} and \eqref{I-I0partII}. Similarly, one has
\begin{equation}\label{II0}
I^{\prime}(v_n)\varphi=I^{\prime}_0(v_n)\varphi+o(1), \;\; \varphi\in X_1.
\end{equation}
Therefore, $(v_n)$ is a Palais-Smale sequence to the functional $I_0$. But, from the same arguments in \cite{CCM02017}, we can show that $I_0$ satisfies Palais-Smale condition for any
$$
0<c<\left(\frac{1}{k+1}-\frac{1}{k^*}\right)\left(\frac{\mathrm{C}^{k}_{N}}{N}\right)^{\frac{k+1}{k^*-k-1}}S^{\frac{N}{2k}}.
$$
Hence, we can assume that $v_n\rightarrow 0$ strongly in $X_1$ and thus $I(v_n)\rightarrow 0$ which contradicts $I(v_n)\rightarrow c_{MP}>0$.
\end{proof}
\begin{lemma} Let $v\in X_1$ be a weak solution for equation \eqref{ourpradial}. Then $v\in C^{2}[0,1]$ and $u(x)=-v(|x|),$ $x\in B$ is a radially symmetric $k$-admissible function which solves the equation \eqref{ourp}. 
\end{lemma}
\begin{proof}
Since $v$ satisfies \eqref{weak-equation}, we can proceed analogously to the prove of Lemma~\ref{Lemma-regular} to get $v$ bounded in $(0,1]$. From \eqref{integral equation- II}, we can write
\begin{equation}\label{e-integral}
(-v^{\prime})^{k}=\frac{\tau}{r^{N-k}}\int_{0}^{r} s^ {N-1}(v^{+})^{k^{*}+s^{\alpha}-1}\ud s.
\end{equation}
Thus, $v\in C^{2}(0,1)$. In addition, since $v$ is bounded in $(0,1]$,  \eqref{e-integral} yields
$$
\lim_{r\rightarrow 0^{+}}v^{\prime}(r)=0.
$$
Hence $v\in C^{1}[0,1]$. Analogously to \eqref{bi-grad} we can write
\begin{equation} \label{Bbi-grad}
\begin{aligned}
v^{\prime\prime}(r) &=\frac{v^{\prime}(r)}{k}\left[-\frac{N-k}{r}+\frac{r^{N-1}|v|^{k^*+r^{\alpha}-1}}{\int_{0}^{r}s^{N-1}|v|^{k^*+s^{\alpha}-1}\ud s}\right],\;\; \forall\; r\in [0,1).
\end{aligned}
\end{equation} 
It follows that $\lim_{r\rightarrow 0^{+}}v^{\prime\prime}(r)$ exists and finally that $v\in C^{2}[0,1]$. Now,  analysis similar to that in the proof of \eqref{v-kadmi} shows $F_{j}(u)\ge 0$ in $B$,  $\forall\; 1\le j\le k$, which ensures that $u$ is a $k$-admissible function. Finally, from \eqref{e-integral}, the function $w=-v$ belongs to $C^{2}[0,1]$ and solves the equation \eqref{Kourpradial}, and consequently $u(x)=w(|x|),$ $x\in B$ solves \eqref{ourp}.
\end{proof}

%\section*{References}
%\bibliographystyle{elsarticle-num}
%\bibliography{mybibfile}
\end{document}